\documentclass[11pt]{article}
 \pdfoutput=1 
\usepackage[letterpaper]{geometry}
\RequirePackage{amsthm,amsmath,amsfonts,amssymb}
\RequirePackage[authoryear]{natbib}
\setcitestyle{authoryear,open={(},close={)}}

\usepackage{latexsym,amssymb,amsmath,amsfonts,amsthm,graphicx,bm, longtable,mathrsfs,bigints,textgreek,upgreek,enumitem,bbm,multirow,authblk}
\RequirePackage[colorlinks,citecolor=blue,urlcolor=blue]{hyperref}

\numberwithin{equation}{section}
\theoremstyle{plain}

\newtheorem{theorem}{Theorem}[section]
\newtheorem{lemma}[theorem]{Lemma}
\newtheorem{proposition}[theorem]{Proposition}
\newtheorem*{result}{Main result}
\theoremstyle{remark}

\newtheorem*{example}{Example}

\usepackage{mathtools, stackrel}
\newtheorem{assump}{Assumption}[section]

\newtheorem{remark}{Remark}[section]
\allowdisplaybreaks

\def\range{\mathscr{R}\,}
\def\nulll{\mathscr{N}\,}
\def\trace{\textsf{\text{Tr}}\,}
\def\E{\mathbb{E}}
\def\calH{\mathcal{H}}
\def\modelconstant{ \vartheta_{g,\beta^*}}
\def\Ltwo{L^2(S)}
\def\tbeta{\tilde{\beta}^*}
\newcommand{\id}{\mathfrak{I}}
\newcommand{\bias}{\textsc{bias}}

\newcommand{\bea}{\begin{eqnarray*}}
\newcommand{\eea}{\end{eqnarray*}}
\newcommand{\be}{\begin{eqnarray}}
\newcommand{\ee}{\end{eqnarray}}
\newcommand{\bsp}{\begin{split}}
\newcommand{\esp}{\end{split}}
\newcommand{\ed}{\end{document}}

\newcommand{\btab}{\begin{tabular}}
\newcommand{\etab}{\end{tabular}}
\newcommand{\bc}{\begin{center}}
\newcommand{\ec}{\end{center}}

\newcommand{\bi}{\begin{itemize}}
\newcommand{\ei}{\end{itemize}}
\newcommand{\bfi}{\begin{figure}}
\newcommand{\efi}{\end{figure}}
\newcommand{\ben}{\begin{enumerate}}
\newcommand{\een}{\end{enumerate}}
\newcommand{\bdes}{\begin{description}}
\newcommand{\edes}{\end{description}}
\newcommand{\bay}{\begin{array}}
\newcommand{\eay}{\end{array}}

\newcommand{\bass}{\begin{assumption}}
\newcommand{\eass}{\end{assumption}}
\newcommand{\bthm}{\begin{theorem}}
\newcommand{\ethm}{\end{theorem}}
\newcommand{\blem}{\begin{lemma}}
\newcommand{\elem}{\end{lemma}}

\def\bco{\iffalse}
\def\cov{{\rm cov}}

\def\trace{{\rm trace}}

\def\cp{\citep}
\newcommand{\floor}[1]{\lfloor #1 \rfloor}

\def\references{\bibliography{bib,1-1-19}
\bibliographystyle{abbrvnat}}

\title{\bf Functional linear and single-index models:\\ A unified approach via Gaussian Stein identity}

\author[1]{Krishnakumar Balasubramanian}
\author[1]{Hans-Georg M\"{u}ller}
\author[2]{Bharath K. Sriperumbudur}
\affil[1]{Department of Statistics, University of California, Davis}
\affil[2]{Department of Statistics,
Pennsylvania State University}
\affil[1]{\texttt{\{kbala,hgmueller\}}@ucdavis.edu}
\affil[2]{\texttt{bks18}@psu.edu}
\date{}

\begin{document}

\maketitle
\begin{abstract}
Functional linear and single-index models are core regression methods in functional data analysis and are widely used for performing regression in a wide range of applications when the covariates are random functions coupled with scalar responses. In the existing literature, however, the construction of associated estimators and the study of their theoretical properties is invariably carried out on a case-by-case basis for specific models under consideration. \textcolor{black}{In this work, \textcolor{black}{assuming the predictors are Gaussian processes}, we provide a unified methodological and theoretical framework for estimating the index in functional linear, and its direction in single-index models}. In the latter case, the proposed approach does not require the specification of the link function. In terms of methodology, we show that the reproducing kernel Hilbert space (RKHS) based functional linear least-squares estimator, when viewed through the lens of an \emph{infinite-dimensional Gaussian Stein's identity}, also provides an estimator of the index of the single-index model. Theoretically, we characterize the convergence rates of the proposed estimators for both linear and single-index models.  Our analysis has several key advantages: (i) it does not require restrictive commutativity assumptions for the covariance operator of the random covariates and the integral operator associated with the reproducing kernel; and (ii) the true index parameter can lie outside of the chosen RKHS, thereby allowing for index \textcolor{black}{misspecification} as well as for quantifying the degree of such index \textcolor{black}{misspecification}. Several existing results emerge as special cases of our analysis. 
\end{abstract}

\section{Introduction}

Functional regression with observed random functions as predictors coupled with scalar responses is one of the core tools of functional data analysis 
\cp{rams:91,morr:15,mull:16:3}. The classical model of functional regression is the functional linear model, which emerges for example, when one assumes a joint Gaussian distribution between the predictor process $X(t)$ and response $Y \in \mathbb{R}$ and is given by 
\begin{align}\label{eq:model}
Y = \int_S X(t) \beta^*(t)\,dt + \epsilon = \langle X, \beta^* \rangle_{\Ltwo} +\epsilon,
\end{align}
where $\epsilon$ is an exogenous additive noise such that $\E\left[\epsilon|X\right] = 0$, $\E[\epsilon^2] = \sigma^2$. In the following, we set $S =[0,1]$. A semi-parametric extension of the above model is the functional single-index model, 
\begin{align}\label{eq:main_model}
Y &= g \left(\int_SX(t)\beta^*(t)\,dt  \right) + \epsilon = g \left(\langle X, \beta^* \rangle_{\Ltwo}  \right) + \epsilon,
\end{align}
for some function $g:\mathbb{R}\to\mathbb{R}$. Following standard terminology, the functional parameter $\beta^*$ is referred to as the index parameter, and the function $g$ as the link function. Note that when $g$ is the identity function,  the single-index model in~\eqref{eq:main_model} becomes the functional linear model.

Given $n$ observations $\{(X_i,Y_i)\}_{1 \leq i \leq n}$ that are independent and identically distributed copies of $(X,Y)$, a fundamental problem is to estimate the index parameter $\beta^*$ in $\eqref{eq:main_model}$. It is worth emphasizing here that for the case of single-index models, an efficient estimator for $\beta^*$ is crucial for subsequently obtaining an estimate of the link function $g$; see, for example,~\cite{mull:11:6}. Hence, our focus is on constructing an estimator of $\beta^*$ that does not require information about the link function. To do so, the interaction between the allowed class of link functions and the distribution of the covariate $X$ becomes crucial, \textcolor{black}{as is also the case in the finite-dimensional case~\citep{yang2017high}}. Indeed this has been well-explored in the case of multivariate single-index models (i.e., when $X \in \mathbb{R}^d$). As will be seen below, this is true in the functional setting as well.

In this work, under a Gaussian process assumption on the covariate $X$, \textcolor{black}{we provide a unified reproducing kernel Hilbert space (RKHS) based framework} for estimating the index in functional linear, the direction of the index in the single-index models, for a class of \emph{unknown} link functions. Specifically, we illustrate that the standard functional linear least-squares estimator also provides an efficient estimator of the index parameter in the single-index model under the Gaussian process assumption. While it might come across as a rather surprising observation at first, it has an elementary justification when the functional linear least-squares estimator is viewed through the lens of \emph{infinite-dimensional} analogs of Gaussian Stein's identity. Similar observations have been made in the multivariate setting (see, for example,~\citealp{brillinger2012generalized, li1989regression, plan2016generalized} and~\citealp{yang2017high}), based on the \emph{finite-dimensional} Gaussian Stein's identity. As our index parameter estimator is agnostic to the choice of the link function, it also naturally handles \textcolor{black}{misspecification} with respect to the link function. \textcolor{black}{Furthermore, compared to existing theoretical results, our analysis handles the case when the true index $\beta^*$ is not necessarily contained in the RKHS that is used for estimation.}

\subsection{Main contributions} \label{sec:contributions}
We now elaborate more on our main contributions in this work.\\

\textit{Methodology:} The RKHS-based functional penalized linear least-squares estimator (for a penalty parameter $\lambda > 0$), given by
    \begin{equation}\label{Eq:estim}
    \hat{\beta}\coloneqq\underset{\beta \in \calH}{\arg\min}~~ \dfrac{1}{n}\sum_{i=1}^n \left[Y_i - \langle\beta,X_i \rangle \right]^2  + \lambda \| \beta \|_\calH^2,
\end{equation}
    was proposed and analyzed in~\cite{yuan2010reproducing} for the linear setting, where $\mathcal{H}$ corresponds to an RKHS with the associated norm $\|\cdot\|_\calH$; see~\citet[Chapter 4]{SteChr2008} for an introduction to RKHS. While the minimization of the regularized objective in \eqref{Eq:estim} is over a possibly infinite-dimensional RKHS $\mathcal{H}$, using the classical ideas of the \emph{representer theorem}, it has been shown in \cite{yuan2010reproducing} that the minimizer of \eqref{Eq:estim} can be computed by solving a finite-dimensional regularized linear inverse problem. In the current work, we illustrate that the above estimator that was designed for linear models, rather surprisingly also serves as a good estimator for the direction of the index in functional single-index models for various link functions $g$, when the random covariates $X_i$ follow a Gaussian process. This provides a unified methodology for estimating the index parameter in functional linear and single-index settings, without regard to the specific nature of the link function $g$,  thereby allowing for its \textcolor{black}{misspecification}. 
    
    Our proposed estimator is based on infinite-dimensional extensions of Gaussian Stein's identity. This goes informally as follows: \textcolor{black}{For a zero-mean Gaussian random element $X$ in a separable Hilbert space with covariance operator $C$, i.e., the linear integral operator with kernel $\cov(X(t),X(\textcolor{black}{s}))$, and for smooth enough real-valued functions $f$, it holds that \begin{equation}\mathbb{E}[Xf(X)] = C \mathbb{E}[\nabla f(X)],\label{Eq:stein}\end{equation} where $\nabla$ is the Fr\'echet derivative. Replacing $f(X)$ in \eqref{Eq:stein} by $g(\langle X,\beta^*\rangle)$, we obtain
    $\mathbb{E}[Xg(\langle X,\beta^*\rangle)] =C\mathbb{E}[\nabla g(\langle X,\beta^*\rangle)],$ with the left hand side~being equivalent to $\mathbb{E}[YX]$ since $\mathbb{E}[YX]=\mathbb{E}[Xg(\langle X,\beta^*\rangle)+X\varepsilon]=\mathbb{E}[Xg(\langle X,\beta^*\rangle)]$ and the right hand side being equivalent to $\modelconstant C\beta^*$ with the constant defined as $\modelconstant:=\mathbb{E}[g'(\langle X,\beta^*\rangle)]$ since $C\mathbb{E}[\nabla g(\langle X,\beta^*\rangle)]=C\mathbb{E}[g'(\langle X,\beta^*\rangle)\beta^*]$, where $g'$ is the derivative of $g$. Therefore, for the single-index model, the Gaussian Stein identity reduces to 
    \begin{equation}\mathbb{E}[YX]=\modelconstant C\beta^*,\label{Eq:Gauss-stein}\end{equation} which in turn reduces to the classical ``functional normal equation" $C\beta^*=\E \left[Y X \right]$ when the model is linear, i.e., $\modelconstant=1$ when $g$ is linear; see, for example, \cite{mull:00:1}. \textcolor{black}{This provides a justification for using the estimator in \eqref{Eq:estim} for estimating the direction of the index in the context of single-index models, as long as $\modelconstant\neq 0$.}} 
To the best of our knowledge, applying this viewpoint and Stein's identity is novel, even in the face of the extensive literature on functional linear and single-index regression models. We also emphasize that the use of Stein's identity in this context enables one to work only with unconditional covariance operators, which is in stark contrast with sufficient dimensionality reduction techniques (for example, sliced inverse regression) that require conditional covariance operators to be estimated. \textcolor{black}{The constraint $\modelconstant\neq 0$ places restrictions on the class of link functions $g$ for which the proposed approach could be used. Examples of link functions for which the constraint holds include logistic functions, odd-powered polynomials, and exponential functions. However, an important function for which it does not hold is the quadratic (or even-powered polynomials). This is due to the fact that the odd moments of Gaussians are zero.}

We note that in the following developments, it is assumed throughout that the predictor process $X$ is a zero-mean Gaussian process that is fully observed. However, the assumption that the process is fully observed may be too strict for some relevant applications, where the functional predictors $X_i$ are observed not continuously but rather intermittently on a dense grid of equidistant design points $(t_1, \dots, t_m)$ on the domain $S$.  Also, the measurements taken at these gridpoints may be contaminated with i.i.d. measurement errors $\epsilon_{ij}$, for  $i=1,\dots,n$ and $j=1,\dots,m,$ i.e., one has $m$ (or more) measurement times $t_{ij}$ that form a dense grid on the domain of the predictor functions $X$. In this situation, the data is observed as $X_{ij}=X_i(t_{ij}) + \epsilon_{ij}$ instead of complete trajectories $X_i$. One can then utilize a uniform convergence result that states that when passing the data $(t_{ij},X_{ij})$ through a local linear smoother that utilizes appropriate bandwidth choices, one may obtain for the resulting curve estimates $\hat{X_i}$,: for any $\varepsilon>0$, $\sup_{t \in S} |\hat{X}_i(t)-X_i(t)| =O_p(m^{-1/(3+\varepsilon)})$. The bound on the right-hand side does not depend on $i$ and can be made arbitrarily small by assuming very dense sampling of individual trajectories and thus the error induced by the pre-smoothing step becomes negligible if $m$ is large enough relative to the sample size $n$.  In order not to detract from the main theme of this paper, we refer for further details about the necessary regularity conditions to Corollary 2 in \cite{mull:21:5}; see also \cite{mull:06:11} and~\cite{hall:07:3} for earlier studies on pre-smoothing of functional data.\\
       
\textit{Theory:} 
    While the true index parameter $\beta^*$ could lie either inside or outside the RHKS under consideration (characterized via interpolation spaces determined by a parameter $\alpha$), the estimator $\hat{\beta}$ in \eqref{Eq:estim} always lies in the RKHS by definition. Previous works \citep{yuan2010reproducing,cai2012minimax, tong2018analysis} relied on the assumption that $\beta^*\in\mathcal{H}$. In this work, we relax this assumption and obtain convergence rates for estimating $\beta^*$ using $\hat{\beta}$ by capturing the interaction between the integral operator $T$ associated with the RKHS and the covariance operator $C$ of the Gaussian process, through 
    the decay behavior of the eigenvalues of the operator $\Lambda:= T^{1/2} C T^{1/2}$ and the alignment of its eigenfunctions to those  of $T$. Our main result (Theorem~\ref{thm:masterthm}), stated informally below, captures the interaction between $T$ and $C$. Specialized versions  (Theorems~\ref{thm:commutativeassumption}, \ref{thm:tcteigendecay}, and \ref{thm:alignedeigensystem}) provide the rate of convergence for $\hat{\beta}$ for both linear and single-index models under appropriate eigenvalue decay assumptions. For the case of a linear model, using a variation of  Theorem~\ref{thm:masterthm} (see Theorem~\ref{thm:predictionmasterthm}), we also provide prediction error results without assuming $\beta^*\in\mathcal{H}$ in Theorem~\ref{thm:predictioncommutativeassumption} and recover the results of \cite{cai2012minimax} in Theorem~\ref{thm:predictionnoncommutativeassumption}.
    \begin{result}[Informal]\label{Thm:informal}
Define $\tbeta:=\modelconstant\beta^*$, where $\modelconstant$ is a model constant that depends on $g$ and $\beta^*$. Suppose $\tbeta \in \range(T^\alpha)$ for $\alpha \in (0,1/2]$, $$\E \left[\left(g(\langle X,\tbeta \rangle) - \langle X,\tbeta\rangle\right)^4\right]<\infty,$$ and   
$\emph{\trace}(C^{1/2}) < \infty$. Then, defining $\Lambda:=T^{1/2}CT^{1/2}$ and $\Lambda_\lambda:=\Lambda+\lambda I$, \textcolor{black}{up to constants, with high probability,} we have  
\begin{align*}
\| \hat{\beta} - \tbeta\| 
\lesssim  \bias(\lambda)+\| \Xi \|^{\frac{1}{4}}\left[  \sqrt{ \frac{N(\lambda)}{n}}
+\lambda \sqrt{\frac{\|\Uptheta\| \emph{\trace}(\Uptheta)}{n}}\right], \nonumber
\end{align*}
where $\Xi\coloneqq T \Lambda^{-2}_\lambda T$, $N(\lambda)\coloneqq \emph{\trace}\left(\Lambda^{-1}_\lambda \Lambda\right)$, $\Uptheta \coloneqq T^{\alpha-\frac{1}{2}}  \Lambda^{-1}_\lambda \Lambda \Lambda^{-1}_\lambda T^{\alpha-\frac{1}{2}}$, and the bias factor is given by $\bias(\lambda) \coloneqq \| T^{1/2}\Lambda^{-1}_\lambda T^{1/2}C\tbeta -\tbeta\|$.
\end{result}

In the above result, the definition of $\tbeta$ allows to treat both single-index and linear models in a unified manner with $\modelconstant=1$ when the model is linear. The case  $\alpha=1/2$ corresponds to $\tbeta\in \mathcal{H}$ and $\alpha<1/2$ corresponds to $\tbeta\in L^2(S)\backslash \mathcal{H}$. The smoothness of the target function $\tbeta$ is captured by $\alpha$, i.e., larger values of $\alpha$ are associated with smoother values of  $\tbeta$, where  $\alpha$ controls the behavior of $\bias(\lambda)$ and $\Uptheta$. The behavior of $\| \hat{\beta} - \tbeta\| $ is also controlled by the decay rate of the eigenvalue of $\Lambda$, which in turn is related to the smoothness of $\mathcal{H}$ and that of the covariance function of the Gaussian process. Finally, the degree of alignment between the eigenfunctions of $\Lambda$ and $T$ controls the behavior of $\bias(\lambda)$ and $\Xi$.

We now highlight the differences and benefits of our results compared to the directly related works of~\cite{yuan2010reproducing}, \cite{cai2012minimax} and \cite{tong2018analysis} that consider only the functional \emph{linear} regression in the RKHS setup. All these works consider the estimator in \eqref{Eq:estim} and predominantly provide convergence results for the easier case of prediction error in the linear setup. While~\cite{yuan2010reproducing} consider the problem of estimation in the linear setting, they make the restrictive assumption that the operators $T$ and $C$ commute.~\cite{cai2012minimax} and  \cite{tong2018analysis} do not make the commutativity assumption, however, they do not provide any results for estimation. Furthermore, all these works assume the true index parameter $\beta^*$ to reside inside the RKHS under consideration. In comparison, our results provide a complete characterization of the rates of estimation without the above assumptions, for both linear and single-index models.


\subsection{Related works}
The functional linear model (i.e., the case when $g$ is the identity function) was derived for the Gaussian case by \cite{gren:50} 
and in various statistical settings was considered by many authors,  with early work by~\cite{engle1986semiparametric}, motivated by analyzing the relation between weather and electricity sales. Subsequent work includes~\cite{rams:91,brum:98,card:99,cuev:02, card:03:2} and~\cite{zhu:14}, to mention a few; reviews include \cite{morr:15} and \cite{mull:16:3}. 

Regarding single-index models, \cite{jame:02} and \cite{muller2005generalized} studied functional versions of generalized linear models that feature a single-index where the latter work included nonparametric estimation of the link function.  In subsequent work, \cite{mull:11:6} studied estimators for general single and multiple index models and provided consistency results, while 
\cite{shang2015nonparametric} developed predictive inferential results for generalized linear models \textcolor{black}{when the true index $\beta^*$ lies in Sobolev RKHS spaces} (which is a special case of the well-specified setting). Their approach follows that of~\cite{yuan2010reproducing} and consequently suffers from similar shortcomings as discussed above. Furthermore, the estimators in the above works depend on the specification of the link function $g$.
Several works, for example,~\cite{hsing2009rkhs, li2017nonlinear, jiang2011functional, li2010deciding} and~\cite{jiang2014inverse}, also considered extension of sufficient dimension reduction methods to the functional data setting, however, only consistency of the estimator is established with no deeper study of convergence rates.


Finally, in the setting of multivariate data, the use of Stein's identity in developing estimation methodology for single and multiple-index models has been well explored. Specifically, we refer to~\cite{li1989regression, plan2016generalized, yang2017high, goldstein2018structured} and~\cite{ goldstein2019non} for the case of single-index models. Similarly, we refer to~\cite{li1991sliced, li1992principal, yang2017estimating} and~\cite{ babichev2018slice} for multiple-index models.

The techniques used in our analysis have a connection to the statistical learning theory literature on analyzing kernel ridge regression methods and linear inverse problems. A comprehensive operator theoretic analysis of kernel ridge regression was provided in~\cite{caponnetto2007optimal} building on the seminal works of~\cite{cucker2002mathematical,de2005learning} and \cite{smale2005shannon}. We also refer the interested reader to the works of~\cite{wu2006learning, smale2007learning,wang2011optimal,  hsu2014random, dicker2017kernel} and~\cite{lin2020optimal} for other related works. \textcolor{black}{Compared to our work, the above works are predominantly focused on excess error bounds for nonlinear regression in the learning theory framework. Another key difference of these works from ours is that they do not involve the covariance operator $C$ and all results are determined by the integral operator $T$ in contrast to ours which depends on the behavior of $T^{1/2}CT^{1/2}$ and its interaction with $T$.} In the context of linear inverse problems,~\cite{blanchard2018optimal} recently used operator-theoretic analysis to also provide estimation error bounds. However, their setting is not directly comparable to our setting of functional linear and single-index model regression.


\subsection{Organization} The rest of the paper is organized as follows. In Section~\ref{sec:method}, we elaborate our unified methodology highlighting the viewpoint obtained by the infinite-dimensional Gaussian Stein's identity. In Section~\ref{sec:theorems}, we present our unified theoretical results for estimating the index parameter. \textcolor{black}{In Section~\ref{sec:interpretation}, we provide examples to interpret and illustrate the assumptions required to derive the main results}. In Section~\ref{sec:predconseq}, the consequences of our results for prediction in the linear setting are highlighted. In Section~\ref{sec:experiments}, numerical simulations for both the Gaussian and non-Gaussian settings are provided. The proofs of all the results are provided in 
Section~\ref{sec:mainproof} and the auxiliary results are provided in Sections~\ref{sec:auxappendix} and~\ref{subsec:etabound}. 

\subsection{Notations}\label{sec:notation}
Unless mentioned explicitly, $\langle \cdot, \cdot \rangle$ and $\|\cdot \|$ refer to $\langle \cdot, \cdot \rangle_{\Ltwo}$ and $\| \cdot \|_{\Ltwo}$,  respectively. We also require the RKHS inner-product and the associated norm sometimes, which we refer to by $\langle \cdot, \cdot \rangle_\calH$ and $\|\cdot \|_\calH$ respectively. For an operator $A$, we denote by $\range(A)$ and $\nulll(A)$, its range space and the null space respectively. We denote the operator norm of $A$ by $\| A\|$. \textcolor{black}{For $x\in H$, $x\otimes x:H \rightarrow H$ is defined as $(x \otimes x)z = x\langle x,z\rangle_H$ for any $z \in H$, where $H$ is a Hilbert space.} We also use $a \lesssim b$ to represent that $a \leq K b$ for a large enough constant $K$. Furthermore, for a random variable $\chi \in \mathbb{R}$ with distribution $P$ and a constant $x$, we use $\chi \lesssim_p x$ to denote the fact that for any $\delta >0$, there exists a positive constant $K_\delta < \infty $ such that $P(\chi \leq K_\delta x) \geq \delta$. 

\section{Methodology}\label{sec:method}
Let $\calH$ be an RKHS with the associated kernel $k: S \times S \to \textcolor{black}{\mathbb{R}}$. Define $\id: \calH \to \Ltwo$, $f\mapsto f$, to be the inclusion operator mapping functions in the RKHS $\calH$ to $\Ltwo$ and $\id^*: \Ltwo  \to  \calH$ to denote the adjoint of $\id$. We also define the following two important operators that arise in our analysis:
\begin{align}\label{eq:importantoperators}
 T \coloneqq \id \id^*: \Ltwo \to \Ltwo \quad \text{and} \quad  C \coloneqq \E[X \otimes X]: \Ltwo \to \Ltwo,
 \end{align}
where $\otimes$ represents the $\Ltwo$ tensor product. \textcolor{black}{Throughout the paper, we assume that $X$ is a centered random element, i.e., $\E[X] = 0_{L_2(S)}$.}

Our goal is to estimate $\beta^*$ in the presence of an unknown link function $g$. First note that one may view the  Gaussian processes $X$ as 
random elements taking values in Hilbert space following a Gaussian measure \citep{rajput1972gaussian,rajput1972gaussianb, grenander2008probabilities,hsin:15}. As discussed in Section~\ref{sec:contributions}, leveraging the version of Stein's identity for Hilbert-valued random elements yields \eqref{Eq:Gauss-stein}, i.e., $\E \left[Y X \right] = \modelconstant C\beta^*$. We refer to \cite{shih2011stein} and~\cite{kuo2011integration} for the details of the infinite-dimensional Gaussian Stein's identity (see \eqref{Eq:stein}) and the associated integration by parts formula and provide a  formal statement in the supplementary material for completeness. Here $\modelconstant$ is a constant depending on the link function $g$ and the index $\beta^*$. The exact value of the constant could be calculated for a given fixed link function $g$, which also fixes the true index parameter $\beta^*$. However, the exact value is irrelevant for our purpose as we focus on estimating the direction of the index parameter (\textcolor{black}{which is the best one could hope for without the knowledge of $g$, due to the lack of identifiability in the model}). Hence, we assume throughout that $g$ is such that $\modelconstant \neq 0$, \textcolor{black}{where $\modelconstant:=\mathbb{E}[g'(\langle X,\beta^*\rangle)]$ is recalled from Section~\ref{sec:contributions}}.  In particular, when $g$ is the identity function, it is easy to see that $\modelconstant=1$. We define $\tbeta\coloneqq \modelconstant \beta^*$ to handle the single-index and linear model in a unified manner and note that 
\begin{align*}
\tbeta \coloneqq \arg\min_{\beta \in \Ltwo} \E \left[ Y - \langle X,\beta\rangle\right]^2.
\end{align*}
This variational formulation is the key step in constructing a regularized estimator as in \eqref{Eq:estim}, whose details are provided below.

Let $(X_1, Y_1), \ldots (X_n, Y_n)$ be $n$ i.i.d. copies of random variables $(X,Y)$. \textcolor{black}{Recalling the definitions in and above~\eqref{eq:importantoperators},} for some $\lambda >0$, our estimator is based on minimizing the penalized least-squares criterion over the RKHS $\calH$,
\begin{align}
\hat{\beta}_{n,\lambda} ={} &  \underset{\beta \in \calH}{\arg\min}~~ \dfrac{1}{n}\sum_{i=1}^n \left[Y_i - \langle\beta,X_i \rangle \right]^2  + \lambda \| \beta \|_\calH^2\nonumber\\
={} &  \underset{\beta \in \calH}{\arg\min}~~ \frac{1}{n}\sum_{i=1}^n \left[Y_i - \langle\id \beta,X_i \rangle \right]^2  + \lambda \| \beta \|_\calH^2\nonumber\\
={} &  \underset{\beta \in \calH}{\arg\min}~~ \frac{1}{n}\sum_{i=1}^n \left[Y_i - \langle\beta,\id^*X_i \rangle_\calH \right]^2  + \lambda \| \beta \|_\calH^2\label{Eq:representer}\\
={} & \underset{\beta \in \calH}{\arg\min}~~ \frac{1}{n}\sum_{i=1}^n \left[Y^2_i + \langle \beta, \left( \id^* X_i\otimes \id^* X_i \right)\beta \rangle_\calH - 2 \langle Y_i \id^*X_i,\beta\rangle_\calH \right] + \lambda \| \beta \|_\calH^2.\label{Eq:squares}
\end{align}
This estimator does not require any knowledge of the link function $g$. Indeed, for the case where $g$ is the identity, this estimator was studied  in~\cite{yuan2010reproducing} for estimation in the functional linear model. As we will demonstrate,   this same estimator continues to be applicable for a general single-index model, under the assumption that $X$ is a Gaussian process, where we focus on the estimation of the direction of the true index parameter $\beta^*$.  

\textcolor{black}{By completing the squares w.r.t.~$\beta$ in \eqref{Eq:squares}, it is easy to verify that} 
\begin{align*}
\hat{\beta}_{n,\lambda} = \left[ \id^* \left(\frac{1}{n}\sum_{i=1}^n X_i\otimes X_i \right) \id + \lambda I \right]^{-1} \id^*\left[\frac{1}{n} \sum_{i=1}^n Y_iX_i \right].
\end{align*}
Defining $\hat{C} \coloneqq \frac{1}{n}\sum_{i=1}^n X_i\otimes X_i$ and $\hat{R} \coloneqq \frac{1}{n}\sum_{i=1}^n Y_iX_i,$ we obtain 
\begin{align}
\hat{\beta}_{n,\lambda} = \left[ \id^* \hat{C} \id + \lambda I \right]^{-1} \id^* \hat{R}. \label{Eq:finite-dim}
\end{align}
In what follows, we denote $\hat{\beta}_{n,\lambda}$ by $\hat{\beta}$ for simplicity. We emphasize that the above form of the estimator is useful for our analysis, while an alternate form that is more useful for implementation purposes is as follows. By applying the representer theorem \citep{Kimeldorf-71,Scholkopf-01} to \eqref{Eq:representer}, $\hat{\beta}\in \text{span}\left\{\int_S k(\cdot,t)X_i(t)\,dt:i=1,\ldots,n\right\},$ i.e., there exists a $\bm{\alpha}:=(\alpha_1,\ldots,\alpha_n)^\top\in\mathbb{R}^n$ such that $\hat{\beta}=\sum^n_{i=1}\alpha_i\int_S k(\cdot,t)X_i(t)\,dt$. Using this in \eqref{Eq:representer} and solving for $\bm{\alpha}$ yields $\bm{\alpha}=(\bm{K}+n\lambda I)^{-1}\bm{y},$ where $\bm{K}\in\mathbb{R}^{n\times n}~\text{with} ~[\bm{K}]_{ij}:=\int_S\int_S k(t,s)X_i(t)X_j(t)\,dt\,ds$ and $\bm{y}=(Y_1,\ldots,Y_n)^\top\in\mathbb{R}^n$. Therefore, $\hat{\beta}$ can be computed by solving a finite-dimensional linear system of size $n$, which is not obvious from the expression in \eqref{Eq:finite-dim}. 

\section{Main results}\label{sec:theorems}
Here we present our main results concerning the rate of convergence of $\hat{\beta}$ to $\tilde{\beta}^*$ in $L^2(S)$. Theorem~\ref{thm:masterthm} (proved in Section~\ref{subsec:master}) is a general result about the behavior of $\Vert \hat{\beta}-\tilde{\beta}^*\Vert$ in terms of certain key operators involving $T$ and $C$. More specialized results are presented in Theorems~\ref{thm:commutativeassumption}, \ref{thm:tcteigendecay} and \ref{thm:alignedeigensystem}, depending on whether $T$ and $C$ commute or not.

\begin{theorem}[Master theorem for estimation]\label{thm:masterthm}
Let $\| T^{-\alpha} \tbeta\| < \infty$, i.e., $\tbeta \in \range(T^\alpha)$ for $\alpha \in (0,1/2]$. Define 
\begin{align}\label{eq:varkappa}
\varkappa \coloneqq \E \left[\left(g(\langle X,\tbeta \rangle) - \langle X,\tbeta\rangle\right)^4\right].
\end{align}
Suppose one of the following conditions hold: (a) $\emph{\trace}(C^{1/2}) < \infty$ and $\varkappa \in (0,\infty)$, (b) $\varkappa=0$ and $\emph{\trace}(C) <\infty$. Define
\begin{align}
\Uptheta &\coloneqq T^\alpha  (CT+\lambda I)^{-1} C (TC+\lambda I)^{-1} T^\alpha,\qquad d(\lambda)  \coloneqq \frac{\emph{\trace}(\Uptheta)}{\|\Uptheta\|}, \label{eq:condition1}\\
\Xi&\coloneqq T (T^{1/2} C T^{1/2} + \lambda I)^{-2} T  ,~~\text{and} \notag\\
N(\lambda)&\coloneqq \emph{\trace}\left[(T^{1/2}CT^{1/2} +\lambda I)^{-1} T^{1/2} C T^{1/2} \right]. \label{eq:condition2}
\end{align}
Then, for
\begin{align}\label{eq:condition3}
\quad\quad\delta \in (0,1/e],\qquad&\qquad n \gtrsim (d(\lambda) \vee \log(1/\delta)), \quad\text{and}\notag \\ 
\frac{\emph{\trace}(T^{1/2}CT^{1/2})}{n}& \lesssim \lambda \lesssim \| T^{1/2} C T^{1/2} \|,
\end{align}
with probability at least $1-3\delta$, we have
\begin{align}\label{eq:masterbound}
\| \hat{\beta} - \tbeta\| &\lesssim  \bias(\lambda)+\| \Xi \|^{\frac{1}{4}}  \sqrt{\frac{(\sigma^2+\sqrt{\varkappa}) N(\lambda)}{n\delta}} \\
&+\lambda\| \Xi \|^{\frac{1}{4}} \left(\left\Vert T^{1/2}CT^{1/2}\right\Vert^{1/2}+\sqrt{\lambda}\right)\| T\|^{\frac{1}{2}-\alpha}\| T^{-\alpha} \tbeta\|\sqrt{\frac{\|\Uptheta\| \emph{\trace}(\Uptheta)}{n}}, \nonumber
\end{align}
where $\bias(\lambda) \coloneqq \| T(CT+\lambda I)^{-1} C\tbeta -\tbeta\|$.
\end{theorem}

\begin{remark}
(i) The assumption $\tbeta \in \range(T^\alpha)$ imposes certain smoothness condition on $\tilde{\beta}^*$. For example, it is well-known \cite[Theorem 4.51]{SteChr2008} that $\tbeta\in\calH$ when $\alpha=\frac{1}{2}$, which we refer to as the \emph{well-specified setting}. This assumption is equivalent to the condition that $\tbeta$ lies in an interpolation space between $L^2(S)$ and $\calH$ with $\alpha$ being the interpolating index.\vspace{1.5mm}\\
(ii) While $\emph{\trace}(C)<\infty$ is guaranteed by the well-definedness of the Gaussian process, Theorem~\ref{thm:masterthm} requires a slightly stronger condition, namely  $\emph{\trace}(C^{1/2}) < \infty$, when $\varkappa\ne 0$.\vspace{1.5mm}\\
(iii) The parameter $\varkappa$ captures the degree of non-linearity of the model. Indeed, $\varkappa=0$ implies $g(\langle X,\tbeta \rangle) =\langle X,\tbeta \rangle$ with probability 1. Conversely, when the model is linear, $\varkappa=0$. \textcolor{black}{For the non-linear case, the condition of $\varkappa < \infty$ is rather mild since it is satisfied by any $g$ that satisfies  $g(x)=o(e^{x^{2+\epsilon}})$ as $x\rightarrow\infty$ for any $\epsilon>0$. Since $\langle X,\tbeta \rangle$ is a zero mean Gaussian random variable, clearly, $\varkappa < \infty$ if $\mathbb{E}[g^4(Z)]<\infty$ which is true if the above condition holds.}

\end{remark}
The following result (proved in Section~\ref{subsec:commutative1}) provides a concrete convergence rate when the operators $T$ and $C$ commute.
\begin{theorem}[Commutative operators]\label{thm:commutativeassumption}
Let $\| T^{-\alpha} \tbeta\| < \infty$ for $\alpha \in (0,1/2]$.  Assume that  the operators $T$ and $C$ commute and have simple eigenvalues (i.e., of multiplicity one) denoted by $\mu_i$ and $\xi_i$ for $i \in \mathbb{N}$, such that
\begin{align}\label{Eq:eigdecay-t-c}
i^{-t} \lesssim \mu_i \lesssim i^{-t}~~\text{and}~~i^{-c} \lesssim \xi_i \lesssim i^{-c},
\end{align}
where $t>1$ and $c>1.$  Suppose one of the following conditions hold: (a) $\varkappa \in (0,\infty)$ and $c >2$, (b) $\varkappa =0$ and $c >1$. Then 
\begin{align}\label{eq:thm2claim}
\| \hat{\beta} - \tbeta\| &\lesssim_p n^{-~\tfrac{\alpha t}{1+c+2t(1-\alpha)}}\quad \text{for} \quad \lambda = n^{-~\tfrac{t+c}{1+c+2t(1-\alpha)}}.
\end{align}
\end{theorem}

\begin{remark}
(i) When $\alpha =1/2$, i.e., $\tbeta \in \calH$ (well-specified case),  we obtain
\begin{align*}
\| \hat{\beta} - \tbeta\| &\lesssim_p n^{-~\tfrac{ t}{2(1+t+c)}},
\end{align*}
which exactly matches the minimax optimal rate obtained in~\cite{yuan2010reproducing} for the functional linear model. Remarkably, this same rate applies in the much more general framework of a single index functional regression
 model when $c>2$ and 
$\varkappa < \infty$.\vspace{1.5mm}\\
(ii) Even for the special case of the functional linear model, Theorem~\ref{thm:commutativeassumption} extends the results of~\cite{yuan2010reproducing} to the misspecified setting, i.e., $\tbeta\in L^2(S)\backslash \calH$, since 
\cite{yuan2010reproducing} only investigated the well-specified setting (i.e., $\tbeta\in\calH$).\vspace{1.5mm}\\
(iii) The term $\alpha$ controls the smoothness of $\tilde{\beta}^*$ with large values of $\alpha$ corresponding to smooth $\tilde{\beta}^*$. Therefore, we should expect the convergence rates to get faster with increasing $\alpha$, which is confirmed by Theorem~\ref{thm:commutativeassumption}. However, based on our current proof technique, Theorem~\ref{thm:commutativeassumption} handles the range of smoothness corresponding to $\alpha \in (0,1/2]$. The case of $\alpha>1/2$ remains open \textcolor{black}{and is an  artifact of our proof technique.} 
\vspace{1.5mm}\\
(iv) The requirement $c>2$ ensures that $\emph{\trace}(C^{1/2}) <\infty$.

\end{remark}
In the following, we relax the assumption of commutativity of $C$ and $T$ and investigate the convergence rates for $\Vert \hat{\beta}-\tbeta\Vert$ by directly exploiting the eigenvalue decay of $T^{1/2} C T^{1/2}$  in Theorem~\ref{thm:tcteigendecay} (proved in Section~\ref{subsec:noncommutative1}). Under additional assumptions about the alignment between the eigenfunctions of $T^{1/2} C T^{1/2}$ and $T$  
faster convergence rates can be obtained. This is the result stated in Theorem~\ref{thm:alignedeigensystem} (proved in Section~\ref{subsec:eigalign}) and the rates there are seen to be faster than those in  Theorem~\ref{thm:tcteigendecay}, while both these convergence rates are slower than those obtained in Theorem~\ref{thm:commutativeassumption} because the commutativity assumption is stronger than these relaxed assumptions.

\begin{theorem}[Noncommutative operators]\label{thm:tcteigendecay}
Let $(\zeta_i)_{i\in\mathbb{N}}$ denote the eigenvalues of $T^{1/2} C T^{1/2}$ with $ i^{-b} \lesssim \zeta_i \lesssim i^{-b}$, for some $b >1$. Suppose $\tbeta \in \range(T^{1/2} (T^{1/2} C T^{1/2})^\nu)$ for $\nu \in (0,1]$ and $\varkappa < \infty$. Then, for
\begin{align}\label{eq:thm3claim}
\| \hat{\beta} - \tbeta\| &\lesssim_p n^{-~\tfrac{ b\nu}{1+b + 2b\nu}} \quad\text{for}\quad \lambda = n^{-~\tfrac{b}{1+b +2b\nu}}.
\end{align}
\end{theorem}

\begin{remark}
(i) Unlike in the commutative case, the results are presented in terms of the eigen decay behavior of $T^{1/2}CT^{1/2}$. When $T$ and $C$ commute, we obtain $b=t+c$.\vspace{1.5mm}\\
(ii) To the best of our knowledge, to date there is no result available in the literature for the estimation error  $\|\hat{\beta}-\tbeta\|$  in the noncommutative setting, even for the special case of functional linear models.\vspace{1.5mm}\\
(iii) The assumption $\tbeta \in \range(T^{1/2} (T^{1/2} C T^{1/2})^\nu)$ implies there is a function  $h \in \Ltwo$ such that $$ T^{1/2} (T^{1/2} C T^{1/2})^\nu h = \tbeta,$$ whence $\tbeta \in \range(T^{1/2}) = \calH$ (i.e., $\alpha =1/2$ in Theorems~\ref{thm:masterthm} and \ref{thm:commutativeassumption}). Therefore, the assumption $\tbeta \in \range(T^{1/2} (T^{1/2} C T^{1/2})^\nu)$ is stronger than assuming $\tbeta\in \range(T^{1/2})$. The key reason for this assumption is to obtain sharper bounds of  $\bias(\lambda)$, thus obtaining non-trivial convergence rates. Indeed, simply assuming $\tbeta\in\range(T^{1/2})$ ensures $\bias(\lambda)\rightarrow 0$ as $\lambda\rightarrow 0$, and consistency of $\hat{\beta}$ can be established, but with no handle on the convergence rate.\vspace{1.5mm}\\
(iv) While it is difficult to grasp the smoothness properties of $\tbeta$ entailed by the condition $\tbeta \in \range(T^{1/2} (T^{1/2} C T^{1/2})^\nu)$ in the noncommutative setting, an understanding of this condition can be gained for the special case where  $T$ and $C$ do commute. In this setting, 
when the eigenvalues of $T$ and $C$ satisfy the conditions of Theorem~\ref{thm:commutativeassumption}, the assumption $\tbeta \in \range(T^{1/2} (T^{1/2} C T^{1/2})^\nu)$ is equivalent to $\tbeta\in\range(T^{\frac{1}{2}+\nu+\frac{c\nu}{t}})\subset \range(T^{1/2})$, which implies that $\tbeta$ is restricted to a smaller subspace of $\calH$. The larger the values of $\nu$ or $\frac{c}{t}$ are, the smaller is this subspace of $\calH$. This means that $\tbeta$ is smoother when $\nu$ increases and when $\nu>0$ as compared to $\nu=0$ (where only $\nu=0$ is actually needed in the commutative case).\vspace{1.5mm}\\
(v) Denoting the eigenfunctions of $T^{1/2}CT^{1/2}$ by $(\phi_i)_{i\in\mathbb{N}}$, in the commutative case, the assumption that $\tbeta \in \range (T^{1/2} (T^{1/2} C T^{1/2})^\nu)$ implies that the bias term behaves as
\begin{align*}
\bias(\lambda)&= \| T(CT+\lambda I)^{-1} C\tbeta -\tbeta\|\nonumber\\
&=\| T(CT+\lambda I)^{-1} CT^{1/2} (T^{1/2} C T^{1/2})^\nu h -T^{1/2} (T^{1/2} C T^{1/2})^\nu h\| \\
  &\leq \left[ \sum_i \left[ \frac{i^{-t -t/2-c-\nu(t+c)}}{i^{-(t+c)} + \lambda } - i^{-t/2 -\nu (t+c)} \right]^2 \langle h,\phi_i \rangle^2 \right]^{1/2} \\
   &\leq  \lambda\left[ \sup_i \frac{i^{-t/2 -\nu(t+c)}}{i^{-(t+c)} + \lambda}\right] \| h\| \leq~\lambda \left[ \lambda ^{\frac{t/2 +\nu(t+c) - (t+c)}{t+c}}\right] =~\lambda^{\nu + \tfrac{t}{2(t+c)}},
\end{align*}
where the last inequality follows from Lemma A.6 
when $(t+c)(1-\nu)\ge t/2$, i.e., $\nu\le \frac{t+2c}{2t+2c}$.
Note that this upper bound is better than $\lambda^{\frac{t}{2(t+c)}}$ when $\alpha =1/2$. Hence, we obtain 
\begin{align*}
\| \hat\beta - \tbeta \|\lesssim_p \frac{\lambda^{-\frac{1+c}{2(t+c)}}}{\sqrt{n}} + \lambda^{\nu + \frac{t}{2(t+c)}},
\end{align*}
where the first term is directly taken from the proof of Theorem~\ref{thm:commutativeassumption} under $\alpha=1/2$. Therefore, 
\begin{align*}
\| \hat\beta - \tbeta \| \lesssim_p n^{-\tfrac{\nu(t+c) + \frac{t}{2}}{2\nu(t+c) +1+c+t}} \quad \text{for} \quad \lambda = n^{-\tfrac{t+c}{2\nu(t+c) +1+t+c}}.
\end{align*}
On the other hand, the bound in Theorem~\ref{thm:tcteigendecay} under the commutativity assumption, i.e., $b=t+c$ yields
\begin{align*}
\| \hat\beta - \tbeta \| \lesssim_p n^{-\tfrac{\nu(t+c)}{2\nu(t+c) +1+c+t}}\quad\text{for}\quad \lambda = n^{-\tfrac{t+c}{2\nu(t+c) +1+t+c}}.
\end{align*}
Thus the bound in Theorem~\ref{thm:commutativeassumption} is better than the one in Theorem~\ref{thm:tcteigendecay}, as expected. 

\end{remark}
As can be seen from the proof of Theorem~\ref{thm:tcteigendecay}, the terms $\|\Xi \|$ and $\bias(\lambda)$ with the assumption 
$\tbeta \in \range(T^{1/2} (T^{1/2} C T^{1/2})^\nu)$ involve interaction terms between $T$ and $T^{1/2} C T^{1/2}$. In contrast, the terms $N(\lambda)$, $\Vert \Theta\Vert$ and $\trace(\Theta)$ are entirely determined by $T^{1/2} C T^{1/2}$. Theorem~\ref{thm:tcteigendecay} ignores the interaction between $T$ and $T^{1/2} C T^{1/2}$. It is of interest to investigate if more refined bounds than those in Theorem~\ref{thm:tcteigendecay} can be obtained by additionally capturing interaction terms.  To this end, let 
$(\zeta_i, \phi_i)$ and $(\mu_i, \psi_i)$ for $i \in \mathbb{N}$ denote the eigensystems of $T^{1/2} C T^{1/2}$ and $T$,  respectively. Then we have
\begin{align*}
\Xi = T(T^{1/2} C T^{1/2}+\lambda I)^{-2}T=T \left[ \sum_i (\zeta_i + \lambda)^{-2} \phi_i\otimes \phi_i + \sum_i \lambda^{-2} \tilde{\phi}_i \otimes \tilde{\phi}_i  \right] T,
\end{align*}
where the $(\tilde{\phi}_i)_i$ span the null space $\nulll(T^{1/2} C T^{1/2})$ of $T^{1/2} C T^{1/2}$. Therefore, 
\begin{align*}
\Vert\Xi\Vert &= \left\Vert\sum_i (\zeta_i + \lambda)^{-2} T\phi_i \otimes T\phi_i + \sum_i  \frac{1}{\lambda^{2}} T\tilde{\phi}_i \otimes T\tilde{\phi}_i\right\Vert \\
&\le \sum_i \frac{\left\Vert T\phi_i\otimes T\phi_i\right\Vert}{(\zeta_i + \lambda)^2}+\frac{1}{\lambda^{2}}\left\Vert \sum_i T\tilde{\phi}_i\otimes T\tilde{\phi}_i\right\Vert
= \sum_i \frac{\left\Vert T\phi_i\right\Vert^2}{(\zeta_i + \lambda)^2}+\frac{1}{\lambda^{2}}\left\Vert \sum_i T\tilde{\phi}_i\otimes T\tilde{\phi}_i\right\Vert \\
&= \sum_i \frac{\left\Vert \sum_j\mu_j\langle \phi_i,\psi_j\rangle\psi_j\right\Vert^2}{(\zeta_i + \lambda)^{2}}+\frac{1}{\lambda^{2}}\left\Vert \sum_i T\tilde{\phi}_i\otimes T\tilde{\phi}_i\right\Vert.
\end{align*}
Note that the first term in the above inequality can be further bounded as follows, 
\begin{align*}
\sum_i \frac{\left\Vert \sum_j\mu_j\langle \phi_i,\psi_j\rangle\psi_j\right\Vert^2}{(\zeta_i + \lambda)^{2}}&=\sum_i\frac{\mu^2_i}{(\zeta_i+\lambda)^2}\sum_j\frac{\mu^2_j}{\mu^2_i}\langle \phi_i,\psi_j\rangle^2 \\ &\le \sum_i\frac{\mu^2_i}{(\zeta_i+\lambda)^2}\sup_i \frac{1}{\mu^2_i}\sum_j\mu^2_j\langle \phi_i,\psi_j\rangle^2.
\end{align*}
Under the assumption that $\sup_i \frac{1}{\mu^2_i}\sum_j\mu^2_j\langle \phi_i,\psi_j\rangle^2<\infty$ (this condition captures the interaction between $T$ and $T^{1/2}CT^{1/2}$ and is naturally satisfied when $T$ and $C$ commute), we obtain
$$\sum_i \frac{\left\Vert \sum_j\mu_j\langle \phi_i,\psi_j\rangle\psi_j\right\Vert^2}{(\zeta_i + \lambda)^{2}}\lesssim \sum_i\frac{\mu^2_i}{(\zeta_i+\lambda)^2}\lesssim \sum_i \frac{i^{-2t}}{(i^{-b}+\lambda)^2}\lesssim \lambda^{-\frac{1+2b-2t}{b}},$$
where the last inequality follows from Lemma A.5  
when $b\ge 2t$ and $b\ge t$, i.e., $b\ge 2t$. Therefore,
$$\Vert \Xi\Vert\lesssim \lambda^{-\frac{1+2b-2t}{b}}+\lambda^{-2}\lesssim \lambda^{-2}$$
since the first term is of smaller order than $\lambda^{-2}$ as $\lambda\rightarrow 0$. 

This shows that because of the interaction between $T$ and  $\nulll(T^{1/2}CT^{1/2})$, it appears that a better bound is not possible for $\Vert \Xi\Vert^{1/4}$, as we showed in the proof of Theorem~\ref{thm:tcteigendecay} (see~\eqref{eq:temp11}) that $\Vert \Xi\Vert^{1/4}\le \lambda^{-1/2}$ without capturing any interaction between $T$ and $\nulll(T^{1/2}CT^{1/2})$. On the other hand, the bound on $\bias(\lambda)$ seems to be improvable. Indeed, note that as $\tbeta = T^{1/2} (T^{1/2} C T^{1/2})^\nu h$ for some $h\in L^2(S)$, we obtain $\bias(\lambda) = \| T^{1/2} (\Lambda + \lambda I)^{-1} \Lambda^{1+\nu} h -  T^{1/2} \Lambda^\nu h \|$ with $\Lambda:=T^{1/2} CT^{1/2}$, 
where the interaction between $T$ and $\nulll(T^{1/2}CT^{1/2})$ does not play a role. These observations lead to the following result (proved in Section~\ref{subsec:eigalign}), which is an improvement over Theorem~\ref{thm:tcteigendecay}. 
\begin{theorem}[Noncommutative operators with alignment of eigenfunctions]\label{thm:alignedeigensystem}
Let $(\zeta_i, \phi_i)$ and $(\mu_i, \psi_i)$ for $i \in \mathbb{N}$, denote the eigensystems of $T^{1/2} C T^{1/2}$ and $T$ respectively. Suppose 
\begin{align*}
   i^{-b} \lesssim \zeta_i \lesssim i^{-b}~\text{and}~ i^{-t} \lesssim \mu_i \lesssim i^{-t}
\end{align*}
for some $b,t >1$ and that the eigenfunctions of $T^{1/2} C T^{1/2}$ and $T$ satisfy
\begin{align}\label{eq:boundedDcondition}
\sup_{i,l} \frac{1}{\mu_i\mu_l} \left| \sum_j \mu_j \langle \phi_i, \psi_j \rangle  \langle \phi_l, \psi_j \rangle\right|^2 < \infty.
\end{align}
Assuming $\varkappa < \infty$ and $\tbeta \in \range(T^{1/2} (T^{1/2} C T^{1/2})^\nu)$ for some $\nu \in \left(0,\frac{1}{2}-\frac{t}{2b}\right]$, we have
\begin{align*}
\| \hat{\beta} - \tbeta\| &\lesssim_p n^{-\tfrac{ b\nu + 
(t-1)/2}{t+b + 2b\nu}}\quad\text{for}\quad \lambda = n^{-\tfrac{b}{t+b +2b\nu}}.
\end{align*}
\end{theorem}

\begin{remark} (i) For $\nu \in (0,\frac{1}{2}-\frac{t}{2b}]$, the rate in Theorem~\ref{thm:alignedeigensystem} is clearly faster than that in Theorem~\ref{thm:tcteigendecay}. \vspace{1.5mm}\\
(ii) When $T$ and $C$ commute, the condition in \eqref{eq:boundedDcondition} is satisfied as
\begin{align*}
\sup_{i,l} \frac{1}{\mu_i\mu_l} \left| \sum_j \mu_j \langle \phi_i, \psi_j \rangle  \langle \phi_l, \psi_j \rangle\right|^2 &= \sup_{i,l} \frac{1}{\mu_i\mu_l} \left| \sum_j \mu_j \langle \phi_i, \phi_j \rangle  \langle \phi_l, \phi_j \rangle\right|^2 \\ &= \sup_{i} \frac{1}{\mu_i^2} \left| \sum_j \mu_j \langle \phi_i, \phi_j \rangle^2 \right|^2 =1.
\end{align*}
Since $b=t+c$ in the commutative setting, by setting $\lambda = n^{-~\tfrac{t+c}{2t+c +2(t+c)\nu}}$, we obtain
\begin{align*}
\| \hat{\beta} - \tbeta\| &\lesssim_p n^{-~\tfrac{ (t+c)\nu + 
(t-1)/2}{2t+c + 2(t+c)\nu}}.
\end{align*}
This rate is still slower than the rate provided by  Theorem~\ref{thm:commutativeassumption}, which is obtained directly under the commutativity assumption since the interaction between $T^{1/2}CT^{1/2}$ and  $T$ is not captured in $\Xi$.
\end{remark}



\section{Interpreting range space conditions on $\tilde{\beta}^*$}\label{sec:interpretation}
In this section, we provide an interpretation of the range space condition $\tilde{\beta}^* \in \mathscr{R}(T^{1/2}(T^{1/2}CT^{1/2})^\nu)$, for $\nu \in (0,1]$, for specific choices of covariance operator $C$ and the kernel $k$ that induces the integral operator $T$. The following result (proved in Section~\ref{subsec:range}) provides a generic characterization of the range space condition, which is elaborated through examples. We consider the case $\nu=1$ for simplicity. 

\begin{proposition}\label{prop:char}
For $x,y \in [0,1]$, suppose that the reproducing kernel $k$ and the covariance function $c$ are given respectively by 
\begin{align*}
k(x,y) = \sum_{i\geq 1} a_i \phi_i(x) \phi_i(y), \quad
c(x,y) = \sum_{m \geq 1} b_m \psi_m(x) \psi_m(x),
\end{align*}
where $a_i\geq 0$ for all $i$, $b_m\geq 0$ for all $m$, $\sum_{i\geq 1} a_i \leq \infty$, $\sum_{m\geq 1} b_m  \leq \infty$ and $(\phi_i)_i$ and $(\psi_m)_m$ form an orthonormal basis of $L^2([0,1])$. Define $\tau_j\coloneqq\sum_i a_i\eta_{ij}^2$ where $\eta_{ij}\coloneqq \sum_{m\geq 1} b_m \theta_{mi} \theta_{mj}$ and $\theta_{mj}\coloneqq\langle \psi_m,\phi_i \rangle$, and assume $\sup_j\tau_j < \infty $. Then it holds that 
\begin{enumerate}
    \item[(i)] The RKHS induced by the kernel $k$ is given by
\begin{align*}
    \mathcal{H} = \left\{ f(x) =\sum_{i\geq 1} f_i \phi_i(x), x \in [0,1]: \sum_{i}\frac{f_i^2}{a_i} < \infty \right\},
\end{align*}
with the associated inner product defined by $\langle f,g\rangle_{\calH} = \sum_i a_i^{-1} f_i g_i$.
\item[(ii)] The space $\mathscr{R}(T^{1/2} (T^{1/2}CT^{1/2}))$ satisfies the inclusion $$\mathscr{R}(T^{1/2} (T^{1/2}CT^{1/2})) \subset \mathcal{\tilde{H}} \subset \calH,$$ where 
\begin{align*}
\mathcal{\tilde{H}} = \left\{ f(x) = \sum_i f_i \phi_i(x), x \in [0,1]: \sum_i \frac{f_i^2}{a_i \tau_i} < \infty\right\},
\end{align*}
is an RKHS induced by the kernel $\tilde{k}(x,y) = \sum_{i\geq 1} a_i \tau_i \phi_i(x) \phi_i(y)$ with inner product $\langle f,g\rangle_{\mathcal{\tilde{H}}} = \sum_{i\geq 1} f_ig_i (\tau_ia_i)^{-1}$. 
\end{enumerate}
\end{proposition}

\begin{remark}
While $T^{1/2}CT^{1/2}$ is a positive self-adjoint operator, its eigenvalues and eigenfunctions are unknown, see~\eqref{eq:tctoperator}. If $\theta_{mi} = \delta_{mi}$, which happens when $\psi_m = \phi_i$ (i.e., in the commutative setting), then we obtain $\eta_{ij} = b_i\delta_{ij}$ and so $T^{1/2}C T^{1/2} = \sum_ia_ib_i \phi_i\otimes \phi_i$, yielding $(a_ib_i, \phi_i)_i$ as the eigensystem of $T^{1/2}CT^{1/2}$, which then implies that $(a^{3/2}_ib_i, \phi_i)_i$ is the eigensystem of $T^{1/2}(T^{1/2}CT^{1/2})$. Therefore, for any $f\in \mathscr{R}(T^{1/2} (T^{1/2}CT^{1/2})^\nu)$, there is a  $h \in L^2([0,1])$ such that we have $f=T^{1/2}(T^{1/2}CT^{1/2})^\nu h$. This implies $f=\sum_i a_i^{\nu+1/2}b_i^\nu h_i\phi_i$ where $h_i = \langle h,\phi_i \rangle$. It is easy to verify that $f \in \mathcal{H}'$, where 
\[
\calH'=\left\{ f(x)=\sum_if_i\phi_i(x), x\in[0,1]: \sum_i\frac{f_i^2}{a_i^{2\nu+1}b_i^{2\nu}} < \infty \right\},
\]
which is an RKHS induced by the kernel $k'(x,y) =\sum_i a_i^{2\nu+1}b_i^{2\nu} \phi_i(x)\phi_i(y)$ since, we have that  $f=T^{1/2}(T^{1/2}CT^{1/2})^\nu h$,
\[
\Vert f\Vert^2_{\mathcal{H}'}=\sum_{i}\frac{a_i^{2\nu+1} b_i^{2\nu} h_i^2}{a_i^{2\nu+1} b_i^{2\nu}} = \sum_i h_i^2 = \| h\|^2 < \infty.
\]
\end{remark}

\begin{remark}
Suppose $\phi_i = \cos(i\pi \cdot)$ and $a_i \propto i^{-2\alpha}$, for some $\alpha \in \mathbb{N}$. Then, for $f\in \calH$, we have $f(x) = \sum_i f_i \phi_i(x) = \sum_i f_i \cos(i\pi x)$, for $x\in[0,1]$ where $\sum_i i^{2\alpha} f_i^2 < \infty$. Note that we have $f^{(\alpha)}(x) = \sum_i\pi^\alpha i^\alpha f_i \cos(i\pi x)$, which implies $\| f^{(\alpha)}\|^2=\pi^{2\alpha} \sum_ii^{2\alpha} f_i^2 = c_1 \|f\|^2_{\calH}$, for some constant $c_1>0$. That is, $\calH$ consists of $\alpha$-times differentiable functions that are square integrable. Suppose $b_i \propto i^{-2\lambda}$ for some $\lambda \in \mathbb{N}$. Then under the conditions of Proposition~\ref{prop:char}, we obtain that $\tilde{\mathcal{H}}$ consists of functions that are $(\alpha+\lambda)$-times differentiable and square-integrable, i.e., the degree of smoothness of  $\mathscr{R}(T^{1/2} (T^{1/2}CT^{1/2}))$ is at least  $\lambda$ more than that of  $\calH$.

\end{remark}

\textcolor{black}{Note that Mercer's theorem allows expansion of the kernel as in Proposition 4.1, wherein $(a_i,\phi_i)_i$ forms the eigensystem of the integral operator, $\mathcal{T}$. Since $S=[0,1]$ and the measure is Lebesgue, the choice of $\phi_i(t)=\cos(i\pi t)$ yields a translation-invariant kernel (see Remark~\ref{rem:spline}) but other choices are possible that yield kernels that are not translation invariant, e.g., $a_i=\frac{1}{i!}$, $\phi_i(x)=x^i$ and $k(x,y)=e^{xy}$.} We now consider concrete examples of covariance kernels and provide interpretations of the result in Proposition~\ref{prop:char}. We let $\phi_i(x)=\cos(i\pi x)$, $x\in[0,1]$.

\begin{example}[Fourier basis]
Suppose $\psi_m = \cos(\omega_m \pi \cdot)$ where $\omega_m=am+b$ for some $a,b \in \mathbb{R}$ such that $\omega_m \notin \mathbb{Z}$ and $m\in \mathbb{N}$. Let $b_m \lesssim m^{-(1+\delta)}$, for some $\delta > 0$. In fact, one can assume without loss of generality that $\omega_m >0$ for all $m \in \{1,2,3,\ldots \}$ or equivalently $a>0$. Then by Lemma A.7, 
we have 
\begin{align*}
    \theta_{mi} &= \langle\psi_m,\phi_i \rangle\\
    & = \int \cos(\omega_m \pi x) \cos(i \pi x) dx \\
    &= \frac{i\pi}{(i\pi)^2 - \omega_m^2\pi^2} \cos(\pi \omega_m) \sin(i\pi) - \frac{\omega_m\pi}{(i\pi)^2 - \omega_m^2 \pi^2} \sin(\pi \omega_m) \cos(i\pi)\\
    &= \frac{\pi \omega_m}{\pi^2 \omega_m^2 - (i\pi)^2} \sin(\pi \omega_m)(-1)^i.
\end{align*}
Furthermore, \begin{align}
    \eta_{ij} = &\sum_m b_m \theta_{mi}\theta_{mj} = \frac{1}{\pi^2}\sum_mb_m \frac{\omega_m}{\omega_m^2-i^2} \frac{\omega_m}{\omega_m^2-j^2} \sin^2(\pi \omega_m) (-1)^{i+j} \notag \\ \stackrel{(*)}{\lesssim}& (ij)^{-\min\left(1,\frac{\delta+1}{2}\right)},\label{Eq:etaij}
\end{align}
where $(*)$ is proved in Appendix B. 
This implies that $\tau_j \lesssim j^{-\min(\delta+1,2)}$ and $\sup_j|\tau_j| < \infty$. Hence, the inclusion $\mathscr{R}(T^{1/2} (T^{1/2}CT^{1/2})) \subset \tilde{\calH} \subset \calH$,  follows from Proposition~\ref{prop:char}, where $\tilde{\calH}$ consists of functions with a degree of smoothness that is by an amount $\min\left(1, \frac{1+\delta}{2}  \right)$ higher than that of the functions in $\calH$.
\end{example}
\begin{remark}\label{Rem:sbm}
The covariance function of a standard Brownian motion is $c(x,y)=\min(x,y)$. It is well known that the eigenvalues and eigenfunctions are  $b_m= \frac{1}{\pi^2(m-\frac{1}{2})^2}$ and $\psi_m(x)=\sqrt{2} \sin \left(\pi (m-\frac{1}{2}) x \right)$. From the above example, it then follows that the space $\mathscr{R}(T^{1/2} (T^{1/2}CT^{1/2}))$ consists of functions that are at least one degree smoother than the functions in $\calH$. 
\end{remark}
\begin{example}[Haar wavelet basis]\label{Exm:Haar}
Let $\psi_m$ be the Haar-wavelet basis function given by
\begin{align}\label{eq:haar}
    \psi_{2^{j}+\ell -1}(x) = \begin{cases}
    +2^{\tfrac{j}{2}}, & \text{for}~ x \in \left[\frac{\ell-1}{2^j}, \frac{\ell-1/2}{2^j}\right] \\
    -2^{\tfrac{j}{2}}, & \text{for }x \in \left[\frac{\ell-1/2}{2^j}, \frac{\ell}{2^j}\right]  \\
    0, & \text{otherwise.}
  \end{cases}
\end{align}
In this case, $|\eta_{ij}| = |\sum_m b_m \theta_{mi} \theta_{mj}| \leq \sum_m b_m |\theta_{mi}||\theta_{mj}|$ with 
\[
\theta_{mi} = \int_0^1  \psi_m(x) \cos(i\pi x) dx  \leq \frac{4}{i\pi} 2^{\floor{\log_2 m}/2},
\]
which follows from Lemma A.8. 
Therefore,
\[
|\eta_{ij}| \leq \sum_m b_m \frac{16}{ij\pi^2} 2^{\floor{\log_2 m}} \leq \frac{16}{ij\pi^2} \sum_m b_m 2^{\log_2 m} = \frac{16}{ij\pi^4} \sum_m b_m m.
\]
Hence, we have
\[
\eta_{ij}^2 \leq \frac{256}{(ij)^2\pi^4} \left(\sum_m b_m m \right)^2
\]
and 
\[
\tau_j = \sum_i a_i \eta_{ij}^2 = \frac{256}{j^2\pi^4} \left(\sum_m b_m m \right)^2\sum_i\frac{a_i}{i^2} \lesssim \frac{1}{j^2}
\]
for any choice of eigenvalues $b_m$ with $\sum_m b_m m < \infty$. Therefore,  $\mathscr{R}(T^{1/2} (T^{1/2}CT^{1/2})) \subset \tilde{\calH} \subset \calH$,  where $\tilde{\calH}$ is an RKHS with one degree of smoothness more than that of $\calH$. 
\end{example}
\begin{remark}\label{rem:spline}
\cite{cai2012minimax} considered the \textcolor{black}{spline} kernel 
\begin{align}\label{eq:simkernel}
k(x,y) = -\frac{1}{3}\left[ B_4(x+y)+B_4(|x-y|)\right],
\end{align}
where $B_4$ is the fourth Bernoulli polynomial. In this case, it can be shown that $a_i=2/(i\pi)^4$ and $\phi_i(x)=\cos(i\pi x)$, and $\calH$ is an RKHS \textcolor{black}{(in fact, a periodic Sobolev space)} of twice differentiable functions which are square integrable on $[0,1]$. For this choice of $a_i$, $\tilde{\calH}$ is the space of thrice differentiable functions which are square integrable on $[0,1]$ for the covariance functions considered in Remark~\ref{Rem:sbm} and Example~\ref{Exm:Haar}.
\end{remark}

\section{Consequences for prediction error in functional linear models}\label{sec:predconseq}
In this section, we investigate the prediction error for the linear predictor $\langle X,\beta^*\rangle$, which is identical to 
$\langle X,\tbeta\rangle$ in the setting of a functional linear model where $\tbeta=\beta^*$.  The prediction error in functional linear models was studied previously in \cite{cai2012minimax}, which showed that a reasonable proxy is  $\| C^{1/2}(\hat{\beta} - \beta^*)\|$, which they analyzed without invoking a commutativity requirement between $T$ and $C$, but under the assumption that $\beta^*\in\calH$.
In the following, we generalize this result as follows. First, in Theorem~\ref{thm:predictionmasterthm} (proved in Section~\ref{subsec:main-pred}), we present a master theorem for the prediction error that does not rely on the assumption  $\beta^*\in \calH$. We specialize this result to non-commutative and commutative settings in Theorems~\ref{thm:predictionnoncommutativeassumption} (proved in Section~\ref{subsec:commutative2}) and \ref{thm:predictioncommutativeassumption} (proved in Section~\ref{subsec:noncommutative2}), respectively, wherein the non-commutative setting recovers the result of \cite{cai2012minimax}, while the commutative setting addresses the scenario of $\beta^*\in L^2(S)\backslash\calH$. 


\begin{theorem}[Master theorem for prediction]\label{thm:predictionmasterthm}
Let $\| T^{-\alpha} \beta^*\| < \infty$ for $\alpha \in (0,1/2]$, i.e., $\tbeta \in \range(T^\alpha)$. Let $\Uptheta$, $d(\lambda)$ be as defined in~\eqref{eq:condition1}, and $N(\lambda)$ be as defined in~\eqref{eq:condition2}. 
Then for $\delta$, $n$ and $\lambda$ satisfying the conditions in~\eqref{eq:condition3}, 
with probability at least $1-3\delta$, we have
\begin{align*}
\| C^{1/2}(\hat{\beta} - \beta^*)\|^2 &\lesssim_p   \frac{\sigma^2 N(\lambda)}{n\delta} + \lambda^2 \| T^{-\alpha}\beta^*\|^{2}\frac{\|\Uptheta\| \emph{\trace}(\Uptheta)}{n}\\
&~~+ \|C^{1/2}T(CT+\lambda I)^{-1} C\beta^* -C^{1/2}\beta^*\|^2.
\end{align*}
\end{theorem}


We now present a specialization of the above result when $T$ and $C$ are not commuting and $\alpha=1/2$.
\begin{theorem}[Prediction for noncommutative operators]\label{thm:predictionnoncommutativeassumption}
Suppose $\beta^* \in \range(T^{1/2})$. Let $(\zeta_i)_{i\in\mathbb{N}}$ denote the eigenvalues of $T^{1/2} C T^{1/2}$ with $i^{-b} \lesssim \zeta_i \lesssim i^{-b}$, for some $b >1$.  Then, for
\begin{align}\label{thm7claim}
\lambda = n^{-\tfrac{b}{1+b}}, \quad\text{we have}\quad \| C^{1/2}(\hat{\beta} - \beta^*)\|\lesssim_p n^{-\tfrac{b}{1+b}}.
\end{align}
\end{theorem}
Compared to Theorem~\ref{thm:tcteigendecay}, Theorem~\ref{thm:predictionnoncommutativeassumption} requires only the weaker assumption that $\beta^* \in \range(T^{1/2})$ instead of $\beta^* \in \range (T^{1/2} (T^{1/2} C T^{1/2})^\nu)$, for some $\nu \in (0,1] $. This is because the prediction error is a weaker notion than the estimation error. The rate obtained in \eqref{thm7claim} was shown to be minimax optimal in \citet{cai2012minimax}.

The following result is another specialization of Theorem~\ref{thm:predictionmasterthm},  where $\beta^*$ is relaxed to lie outside $\calH$ but $T$ and $C$ are assumed to commute. Thus compared to Theorem~\ref{thm:predictionnoncommutativeassumption}, Theorem~\ref{thm:predictioncommutativeassumption} considers the alternate setting with a weaker assumption on  $\beta^*$ and a stronger assumption on $T$ and $C$.

\begin{theorem}[Prediction for commutative operators]\label{thm:predictioncommutativeassumption}
Let $\| T^{-\alpha} \beta^*\| < \infty$ for $\alpha \in (0,1/2]$.  
Suppose the operators $T$ and $C$ commute and have simple eigenvalues (i.e., of multiplicity one) denoted by $\mu_i$ and $\xi_i$ for $i \in \mathbb{N} $, such that for some $t >1$ and $c>1$, they satisfy the condition in~\eqref{Eq:eigdecay-t-c}. Then, by setting $\lambda$ as in~\eqref{eq:thm2claim}, we obtain 
\begin{align}\label{eq:thm6claim}
\| C^{1/2}(\hat{\beta} - \beta^*)\| &\lesssim_p n^{-~\tfrac{2\alpha t+c}{1+c+2t(1-\alpha)}}.
\end{align}
\end{theorem}
The above result extends the results of~\cite{yuan2010reproducing} to the case where $\beta^*$ does not necessarily lie in the RKHS $\calH$. When $\alpha=1/2$,  we recover the corresponding result from~\cite{yuan2010reproducing}, which matches with \eqref{thm7claim}.

\section{Numerical simulations}\label{sec:experiments}
In this section, using numerical simulations, we examine the robustness of the proposed method for non-Gaussian predictors while validating the presented theoretical results for Gaussian predictors. To this end, we let $\epsilon \sim N(0,1)$ in \eqref{eq:model} and follow the setup in~\cite{hall2007methodology} and~\cite{cai2012minimax} for $\beta^*(t)$ and $X(t)$, wherein $\beta^*(t)\coloneqq \sum_{j=1}^{50} 4(-1)^{(k+1)}k^{-2} \phi_k(t)$ with $\phi_1(t)\equiv 1$ and $\phi_{k+1}(t) = \sqrt{2} \cos(k\pi\,t),\,t\in[0,1]$, for $k\geq 1$, and $X(t)=\sum_{k=1}^{50} (-1)^{(k+1)} k^{-2}Z_k \phi_{k}(t)$, with $Z_k$ being one of the following:
\begin{itemize}[noitemsep,leftmargin=0.2in]
    \item \textit{Gaussian:} $Z_k \sim N(0,1)$ which leads to Gaussian process predictors, satisfying our assumptions. 
    \item \textit{Non-Gaussian:} $Z_k \sim \textsc{unif}[-3,3]$, which does not satisfy our assumptions. 
\end{itemize}
Following~\cite{yang2017estimating}, we consider four choices for the link function: \textit{(i)} linear, \textit{(ii)} $g_1(u)=g(u) = 3 u + 10 \sin(u) $, \textit{(iii)} $g_2(u)=g(u) = \sqrt{2} u + 4 \exp(-2 u^2)$, and \textit{(iv)} logistic, and following~\cite{cai2012minimax}, the kernel is chosen to be the one defined in~\eqref{eq:simkernel}. Using the above, our estimator is constructed as described in the paragraph below~\eqref{Eq:finite-dim}. In the construction of the kernel matrix, we select a grid size of $100$ to approximate the integral over $S=[0,1]$. Since our main focus is to compare the estimated direction to the true direction without explicitly providing consideration to $\modelconstant$, we consider the cosine distance between the estimator and the truth, defined as $ 1- \| \hat \beta \|^{-1}\| \beta^* \|^{-1}   \langle \hat \beta(t), \beta^*(t) \rangle$ as a measure of the quality of estimation. Finally, since our main purpose is only to demonstrate the robustness of the proposed approach to any deviations from the assumptions, we set the tuning parameter manually to the best-performing one.  

\begin{figure}
\vspace{1mm}
    \centering
    \includegraphics[width=0.4\linewidth]{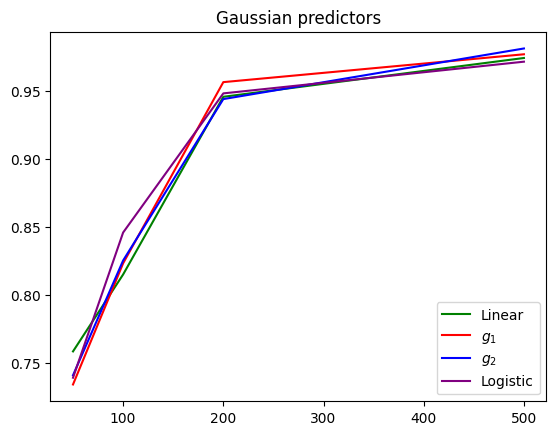}        \includegraphics[width=0.4\linewidth]{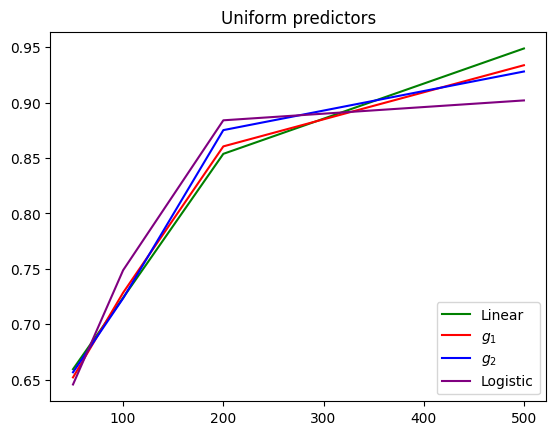}
    \caption{Cosine distance versus sample size: The figure on the left corresponds to the case of Gaussian predictors and figure on the right correspond to the case of uniform predictors. }
    \label{fig:enter-label}
\end{figure}

In Figure~\ref{fig:enter-label}, we show the cosine distance averaged over 1000 simulations. From the results, the following two observations that support our methodology and theory can be made. First, note that despite the true model being not necessarily a linear model, the linear estimator succeeds in estimating the direction. Second, although our methodology and theoretical results are derived under the Gaussian process assumption, the proposed approach works equally well with non-Gaussian predictors.~\cite{goldstein2019non} used a non-Gaussian version of finite-dimensional Stein's identity to explain this observation for the Euclidean setting. In the infinite-dimensional setting, non-Gaussian Stein's identities are not well explored. Deriving such results and providing theoretical support for this empirical observation are left as intriguing future work. 



\section{Proofs} 
\label{sec:mainproof}
In this section, we provide the proof of the results of Sections~\ref{sec:theorems}-\ref{sec:predconseq}. 

\subsection{Proof of Theorem~\ref{thm:masterthm}} \label{subsec:master}
The proof proceeds by first decomposing $\| \hat{\beta} - \tbeta\|$ into several terms which are subsequently upper-bounded individually. \textcolor{black}{Recall $\hat{C} \coloneqq \frac{1}{n}\sum_{i=1}^n X_i\otimes X_i$ and $\hat{R} \coloneqq \frac{1}{n}\sum_{i=1}^n Y_iX_i$}. Define
\begin{align*}
A \coloneqq \id^* C \id,\,\, \hat{A} \coloneqq \id^* \hat{C} \id,\,\,B = \id^* \id,\,\,\text{and}
\end{align*}
\begin{align}\label{eq:betalambda}
{\beta}_{\lambda} = \left[ A + \lambda I \right]^{-1} \id^* \E[YX]  = \left[ A + \lambda I \right]^{-1}  \id^*C\tbeta,
\end{align}
where the last equality in $\beta_\lambda$ follows from Stein's identity. 
By the definition of $ \beta_\lambda$ in~\eqref{eq:betalambda}, we have 
\begin{align*}
\hat\beta -\beta_\lambda &= (\hat A + \lambda I)^{-1} \left[ \id^*\hat R - (\hat A+ \lambda I) \beta_\lambda \right]
=(\hat A + \lambda I)^{-1} \left[ \id^*\hat R - \hat A\beta_\lambda - \lambda  \beta_\lambda \right] \\
&=(\hat A + \lambda I)^{-1} \left[ \id^*\hat R - \hat A\beta_\lambda + A  \beta_\lambda -  \id^*C\tbeta \right] \\ 
&=(\hat A + \lambda I)^{-1} \left[ \id^*\hat R -  \id^* \hat{C} \id\beta_\lambda +  \id^* {C} \id  \beta_\lambda -  \id^*C\tbeta \right] \\ 
&=(\hat A + \lambda I)^{-1} \left[ \id^*\hat R -  \id^* \hat{C} \tbeta +  \id^* \hat{C}  \tbeta +  \id^*C\id\beta_\lambda - \id^*C\tbeta -\id^*\hat{C}\id \beta_\lambda \right] \\ 
&=(\hat A + \lambda I)^{-1} \left[ \id^*\hat R -  \id^* \hat{C} \tbeta  +  \id^* (C- \hat{C}) (\id\beta_\lambda - \tbeta )\right].
\end{align*}
Based on the above identity, we then have
\begin{align*}
\| \hat{\beta} - \tbeta\| &= \| \id \hat{\beta} - \tbeta\|= \left\|\id( \hat{\beta} - \beta_\lambda) + \id \beta_\lambda - \tbeta\right \| \nonumber \\
&\leq \| \id(\hat\beta -\beta_\lambda) \| + \| \id\beta_\lambda -\tbeta\|
= \| B^{1/2}(\hat\beta -\beta_\lambda) \|_\calH +  \| \id\beta_\lambda -\tbeta\|  \nonumber  \\
& = \left\| B^{1/2}(\hat A + \lambda I)^{-1} \left[ \id^*\hat R -  \id^* \hat{C} \tbeta  +  \id^* (C- \hat{C}) (\id\beta_\lambda - \tbeta )\right]\right \|_\calH 
+  \| \id\beta_\lambda -\tbeta\|.  \nonumber \end{align*}
Since
\begin{align*}
&B^{1/2}(\hat A + \lambda I)^{-1} \left[ \id^*\hat R -  \id^* \hat{C} \tbeta  +  \id^* (C- \hat{C}) (\id\beta_\lambda - \tbeta )\right]\\
&=B^{1/2}(A + \lambda I)^{-1/2} ( A + \lambda I)^{1/2}  (\hat A + \lambda I)^{-1/2}  ( \hat A + \lambda I)^{-1/2}( A + \lambda I)^{1/2}  \nonumber  \\ &\qquad\qquad \cdot ( A + \lambda I)^{-1/2}\left[ \id^*\hat R -  \id^* \hat{C} \tbeta  +  \id^* (C- \hat{C}) (\id\beta_\lambda - \tbeta )\right]
\end{align*}
we obtain
\begin{align*}
&\left\| B^{1/2}(\hat A + \lambda I)^{-1} \left[ \id^*\hat R -  \id^* \hat{C} \tbeta  +  \id^* (C- \hat{C}) (\id\beta_\lambda - \tbeta )\right]\right \|_\calH\\
&\leq\| B^{1/2}(A + \lambda I)^{-1/2} \| \cdot \| ( A + \lambda I)^{1/2}   (\hat A + \lambda I)^{-1/2} \| \nonumber\\
&\qquad \cdot \|  ( \hat A + \lambda I)^{-1/2} ( A + \lambda I)^{1/2} \|  \nonumber \\
&\qquad\qquad \cdot\left\| ( A + \lambda I)^{-1/2} \left[ \id^*\hat R -  \id^* \hat{C} \tbeta  +  \id^* (C- \hat{C}) (\id\beta_\lambda - \tbeta )\right] \right\|_\calH,
\end{align*}
therefore resulting in
\begin{align}
\label{eq:maindecomp}
\| \hat{\beta} - \tbeta\|
& \leq  \underbrace{\| B^{1/2}(A + \lambda I)^{-1/2} \|}_{\texttt{Term 1}}\cdot \underbrace{\| ( A + \lambda I)^{1/2}   (\hat A + \lambda I)^{-1/2} \|}_{\texttt{Term 2}}\nonumber\\
&\qquad\cdot \underbrace{\|  ( \hat A + \lambda I)^{-1/2} ( A + \lambda I)^{1/2} \|}_{\texttt{Term 3}}   
 \cdot\left[ \underbrace{ \left\| ( A + \lambda I)^{-1/2} \left[ \id^*\hat R -  \id^* \hat{C} \tbeta \right] \right \|_\calH}_{\texttt{Term 4}}\right.\nonumber\\
 &\qquad\qquad\left.+  \underbrace{\left\| ( A + \lambda I)^{-1/2}  \id^* (C- \hat{C}) (\id\beta_\lambda - \tbeta ) \right\|_\calH}_{\texttt{Term 5}} \right]
+ \underbrace{ \| \id\beta_\lambda -\tbeta\|}_{\texttt{Term 6}}. 
\end{align}

\subsubsection*{Bounding \texttt{Term 1}}
\begin{align*}
\|B^{1/2} (A+\lambda I)^{-1/2}\|^2 
& =\|(A+\lambda I)^{-1/2} B (A+\lambda I)^{-1/2}  \|  = \|B^{1/2} (A+\lambda I)^{-1} B^{1/2} \|\\
& = \| (A+\lambda I)^{-1/2} B^{1/2} \|^2 \leq  \| (A+\lambda I)^{-1} B \| = \| (\id^*C\id + \lambda I)^{-1} \id^* \id\| \\
& = \| \id^* (CT+\lambda)^{-1}\id\| 
 = \| \id^* T^{-1/2} (T^{1/2} C T^{1/2} + \lambda I)^{-1} T^{1/2} \id \| \\
& \stackrel{(*)}{=} \| \id^* T^{1/2} (\Lambda + \lambda I)^{-1} T^{-1/2} \id \id^* T^{-1/2} (\Lambda + \lambda I)^{-1} T^{1/2} \id   \|^{1/2} \\
& = \| \id^* T^{1/2} (T^{1/2} C T^{1/2} + \lambda I)^{-2} T^{1/2} \id  \|^{1/2}\\
& = \|   (T^{1/2} C T^{1/2} + \lambda I)^{-1} T^{1/2} T T^{1/2}  (T^{1/2} C T^{1/2} + \lambda I)^{-1} \|^{1/2} \\
& = \|   (T^{1/2} C T^{1/2} + \lambda I)^{-1}  T^2  (T^{1/2} C T^{1/2} + \lambda I)^{-1} \|^{1/2} \\
& = \| T (T^{1/2} C T^{1/2} + \lambda I)^{-2} T\|^{1/2},
\end{align*}
where $\Lambda:=T^{1/2} C T^{1/2}$ in $(*)$ and the only step with the inequality in the above sequence of calculations, follows from Cordes' inequality~\citep{cordes1987spectral}.  Hence, we have
\begin{align}\label{eq:term1bound}
\|B^{1/2} (A+\lambda I)^{-1/2}\| \leq  \| T (T^{1/2} C T^{1/2} + \lambda I)^{-2} T\|^{1/4}. 
\end{align}

\subsubsection*{Bounding \texttt{Term 2} and \texttt{Term 3}:}

First note that
\begin{align*}
\| ( A + \lambda I)^{1/2}   (\hat A + \lambda I)^{-1/2} \|^2 & = \| ( \hat A + \lambda I)^{-1/2} (A+\lambda I)   (\hat A + \lambda I)^{-1/2} \| \\
& = \| ( A + \lambda I)^{1/2}   (\hat A + \lambda I)^{-1} (A+\lambda)^{1/2}\|
 = \| (\hat A + \lambda I)^{-1/2}   ( A + \lambda I)^{1/2} \|^2,
\end{align*}
which implies that \texttt{Term 2}~$=$~\texttt{Term 3}. Next, note that
\begin{align*}
&\| ( A + \lambda I)^{1/2}   (\hat A + \lambda I)^{-1/2} \|^2  = \| ( A + \lambda I)^{1/2}   (\hat A + \lambda I)^{-1} (A+\lambda)^{1/2}\| \\
& = \left\| \left[ I - (A+\lambda I)^{-1/2} (A -\hat{A}) (A+\lambda I)^{-1/2} \right]^{-1}  \right\| 
 \leq \frac{1}{1-\left\| (A+\lambda I)^{-1/2} (A -\hat{A}) (A+\lambda I)^{-1/2} \right\|}.
\end{align*}
Define 
\begin{align*}
\Sigma & \coloneqq (A+\lambda I)^{-1/2} A (A+\lambda I)^{-1/2}
 =  (A+\lambda I)^{-1/2} \id^*C\id (A+\lambda I)^{-1/2} \\
& = \E\left[ (A+\lambda I)^{-1/2} \id^* (X \otimes X) \id (A+\lambda I)^{-1/2} \right],
\end{align*} 
and
\begin{align*}
\hat\Sigma & \coloneqq (A+\lambda I)^{-1/2} \hat A (A+\lambda I)^{-1/2}
=  (A+\lambda I)^{-1/2} \id^*\hat C\id (A+\lambda I)^{-1/2} \\
& = \frac{1}{n} \sum_{i=1}^n  (A+\lambda I)^{-1/2} \id^* (X_i \otimes X_i) \id (A+\lambda I)^{-1/2}.
\end{align*} 
This yields $\| (A+\lambda I)^{-1/2} (A -\hat{A}) (A+\lambda I)^{-1/2} \| =  \|\hat \Sigma -\Sigma\| $. Therefore by Theorem A.3, 
for any 
\begin{align}\label{eq:temp1}
n \geq ( r(\Sigma) \vee \tau) \quad \text{and} \quad \tau \geq1,
\end{align} with probability at least $1-e^{-\tau}$, we have
\begin{align}\label{Eq:bound}
 \|\hat \Sigma -\Sigma\|~ \leq ~K_1\| \Sigma\|~\frac{\sqrt{r(\Sigma)} + \sqrt{\tau}}{\sqrt{n}}\leq K_1\frac{\sqrt{r(\Sigma)} + \sqrt{\tau}}{\sqrt{n}},
\end{align}
where $K_1$ is a universal constant and we used $\Vert \Sigma\Vert=\Vert (A+\lambda I)^{-1/2} A (A+\lambda I)^{-1/2}\Vert\le 1$. The effective rank $r(\Sigma)$ satisfies
$$
r(\Sigma) \leq \frac{\trace(\Sigma)}{\| \Sigma\|} = \frac{\trace((A+\lambda I)^{-1/2} A (A+\lambda I)^{-1/2} )}{\| (A+\lambda I)^{-1/2} A (A+\lambda I)^{-1/2} \|}\\
= \frac{\trace((A+\lambda I)^{-1} A )}{\sup_i \left[\frac{\lambda_i(A)}{\lambda + \lambda_i(A)} \right]},
$$
with $\lambda_i(A)$ denoting the $i^{th}$ eigenvalue of operator $A$. Observe that
\begin{align*}
 \trace((A+\lambda I)^{-1} A) \leq \| (A+\lambda I)^{-1}\| \trace(A) \leq \frac{\trace(A)}{\lambda}.
\end{align*}
Furthermore,
\begin{align*}
\sup_i \left[\frac{\lambda_i(A)}{\lambda + \lambda_i(A)} \right] \geq \frac{\sup_i \lambda_i(A)}{\lambda + \|A \|} = \frac{\|A \|}{\lambda + \|A\|}.
\end{align*}
Hence,
\begin{align}\label{Eq:rSigma}
r(\Sigma)  \leq & \left( \frac{\lambda + \| A \|}{\| A \|}\right) \frac{\trace(A)}{\lambda}  \leq \left( 1+ \frac{\lambda }{\| A \|}\right) \frac{\trace(A)}{\lambda} 
= \left( \frac{1}{\lambda} + \frac{1}{\| A \|} \right) \trace(A)\leq \frac{2 \trace(A)}{\lambda},
\end{align}
where the last inequality holds since $\lambda \leq \| A \|$. Using \eqref{Eq:rSigma} in \eqref{Eq:bound}, we obtain that with probability at least $1-e^{-\tau}$,
\begin{align*}
\|\hat \Sigma -\Sigma\| \leq K_1 \left(\sqrt{\frac{2\trace(A)}{\lambda n}} + \sqrt{\frac{\tau}{n}}\right) \leq  \frac{1}{2}
\end{align*}
if 
\begin{align}\label{eq:temp2}
\lambda \geq \frac{32K^2_1 \trace(A)}{n}~\quad\text{and}\quad n \geq 16\tau K_1^2.
\end{align}
Combining~\eqref{eq:temp1} and~\eqref{eq:temp2}, we see that we require $\tau \geq 1$, $n\geq \tau$ and $n \geq r(\Sigma)$, which are satisfied as long as $n \geq 2\trace(A) /\lambda$ or equivalently $\lambda \geq 2\trace(A) /n$. Finally, note that we have 
\begin{align*}
\trace(A) = \trace(\id^* C\id) = \trace(CT) = \trace(T^{1/2} C T^{1/2}),
\end{align*}
and
\begin{align*}
\|A\| = \| \id^* C \id \| = \| C^{1/2} T C^{1/2} \| = \| T^{1/2} C T^{1/2} \|.
\end{align*}
Hence, as long as 
\begin{align}\label{eq:temp3}
(2\vee 32K_1^2) &\frac{\trace(T^{1/2} C T^{1/2})}{n} \leq \lambda \leq \| T^{1/2} C T^{1/2}\|, \nonumber\\
\delta \leq \frac{1}{e}\quad&\text{and}\quad n \geq (1\vee16 K_1^2) \ln\left(\frac{1}{\delta}\right),
\end{align}
we have with probability at least $1-\delta$, $\| \hat \Sigma - \Sigma \| \leq 1/2$. Hence under the conditions in~\eqref{eq:temp3}, 
\begin{align}\label{eq:term23bound}
\| (A+\lambda I)^{1/2} (\hat A + \lambda I)^{-1/2} \| \leq \sqrt{\frac{1}{1-\|\hat \Sigma - \Sigma\|}} \leq \sqrt{2}. 
\end{align}

\subsubsection*{Bounding \texttt{Term 4}}

Define 
$
Z_i := (A+\lambda I)^{-1/2} \left[ \id^* Y_i X_i - \id^* (X_i \otimes X_i) \tbeta \right]
$ so that
$$\E[Z_i] = (A+\lambda I)^{-1/2} \left[ \id^* \E[Y X] - \id^* C \tbeta \right] =0.$$ 
By Chebyshev's inequality for Hilbert-valued random variables (see Lemma A.2), 
with probability at least $1-\delta$, we have  
\begin{align}
\left\| (A+\lambda I)^{-1/2} \left[ \id^* \hat R - \id^* \hat C \tbeta \right] \right\|_\calH 
\leq  \sqrt{\frac{ \E\left[\left\| (A+\lambda I)^{-1/2} \left[ \id^* YX - \id^* (X \otimes X) \tbeta\right] \right\|^2_\calH\right]}{n\delta}},\label{Eq:Cheby}
\end{align}
where
\begin{align}
&\E\left[\left\| (A+\lambda I)^{-1/2} \left[ \id^* YX - \id^* (X \otimes X) \tbeta\right]\right\|^2_\calH\right] 
=  \E\left[\left\| (A+\lambda I)^{-1/2} \left[ \id^* (Y - \langle X,  \tbeta \rangle ) X \right] \right\|^2_\calH \right] \nonumber\\
&=  \E\left[(Y - \langle X,  \tbeta \rangle )^2 \left\| (A+\lambda I)^{-1/2}  \id^*  X  \right\|^2_\calH \right] \nonumber \\
&=  \E\left[(Y - \langle X,  \tbeta \rangle )^2 \left\langle  (A+\lambda I)^{-1/2} \id^* X, (A+\lambda I)^{-1/2} \id^* X\right\rangle_\calH \right] \nonumber \\
&=  \E\left[(Y - \langle X,  \tbeta \rangle )^2 \trace\left( (A+\lambda I)^{-1} \id^* (X\otimes X) \id \right) \right] \nonumber \\
&=  \E\left[\left(Y - g(\langle X,  \tbeta \rangle) + g(\langle X,  \tbeta \rangle) - \langle X,  \tbeta \rangle \right)^2\right.\nonumber\trace\left( (A+\lambda I)^{-1} \id^* (X\otimes X) \id \right) \Big] \nonumber \\
&\leq  2 \E\left[ \left\{ \epsilon^2 + ( g(\langle X,  \tbeta \rangle) - \langle X,  \tbeta \rangle )^2 \right\}\trace\left( (A+\lambda I)^{-1} \id^* (X\otimes X) \id \right) \right] \nonumber \\
& = 2 \E [\epsilon^2]\, \E\left[ \trace\left( (A+\lambda I)^{-1} \id^* (X\otimes X) \id \right) \right]\nonumber\\
&\qquad\qquad+ 2 \E\left[  ( g(\langle X,  \tbeta \rangle) - \langle X,  \tbeta \rangle )^2  \trace\left( (A+\lambda I)^{-1} \id^* (X\otimes X) \id \right) \right] \nonumber\\ 
 &=  2 \E [\epsilon^2] \trace\left( (A+\lambda I)^{-1} \id^* C \id \right)\nonumber\\
 &\qquad\qquad+ 2 \E\left[  ( g(\langle X,  \tbeta \rangle) - \langle X,  \tbeta \rangle )^2  \trace\left( (A+\lambda I)^{-1} \id^* (X\otimes X) \id \right) \right] \nonumber \\
&\leq  2 \sigma^2 \trace\left( (A+\lambda I)^{-1} A \right) + 2\sqrt{\varkappa} \sqrt{\E\left[ \trace^2 \left( (A+\lambda I)^{-1} \id^* (X\otimes X) \id \right) \right]},\label{eq:temp4} 
\end{align}
where we recall that $\varkappa$ is defined in~\eqref{eq:varkappa}. Recalling the definition of $N(\lambda)$ from~\eqref{eq:condition2}, we have 
\begin{align*}
\trace\left( (A+\lambda I)^{-1} A \right) 
&=\trace\left( (\id^*C\id +\lambda I)^{-1} \id^* C\id \right)  = \trace\left( \id^* (C\id\id^*+\lambda I)^{-1} C\id \right) \nonumber\\
 &= \trace\left( T(CT+\lambda I)^{-1} C\right) = \trace\left( CT(CT+\lambda I)^{-1} \right) \nonumber\\
&= \trace\left( CT^{1/2} (CT^{1/2} +\lambda T^{-1/2} )^{-1} \right) \nonumber\\
 &= \trace\left( T^{1/2} CT^{1/2} (T^{1/2} CT^{1/2} +\lambda I )^{-1} \right) = N(\lambda).
\end{align*}
Furthermore,
\begin{align*}
&\trace \left( (A+\lambda I)^{-1} \id^* (X\otimes X) \id \right)
= \trace \left( \id (A+\lambda I)^{-1} \id^* (X\otimes X) \right) \\
&= \trace \left( \id (\id^* C \id +\lambda I)^{-1} \id^* (X\otimes X) \right) 
= \trace \left( \id\id^* ( C \id \id^* +\lambda I)^{-1} (X\otimes X) \right) \\
&= \trace \left( T ( C T +\lambda I)^{-1} (X\otimes X) \right) 
= \trace \left( T^{1/2} ( T^{1/2}C T^{1/2} +\lambda I)^{-1} T^{1/2} (X\otimes X) \right) \\
&=  \left\langle X, T^{1/2} ( T^{1/2}C T^{1/2} +\lambda I)^{-1} T^{1/2} X \right\rangle.
\end{align*}
Therefore,
\begin{align}
&\E\left[\trace^2 \left( (A+\lambda I)^{-1} \id^* (X\otimes X) \id \right) \right]
 = \E \left[  \left\langle X, T^{1/2} ( T^{1/2}C T^{1/2} +\lambda I)^{-1} T^{1/2} X \right\rangle^2   \right]\nonumber \\
 &\stackrel{(*)}{\leq} 3 \trace^2 \left(T^{1/2} ( T^{1/2}C T^{1/2} +\lambda I)^{-1} T^{1/2} C  \right) 
 \leq 4 N^2(\lambda), \label{eq:temp5}
 \end{align}
 where $(*)$ follows from Lemma A.4. 
 Combining \eqref{eq:temp4} and \eqref{eq:temp5} in \eqref{Eq:Cheby}, we obtain with  probability at least $1-\delta$,
\begin{align}\label{eq:term4bound}
\left\| (A+\lambda I)^{-1/2} \left[ \id^* \hat R - \id^* \hat C \tbeta \right] \right\|_\calH \leq \sqrt{\frac{(2\sigma^2 + 4 \sqrt{\varkappa}) N(\lambda)}{n\delta}},
\end{align}
under the assumption that $\trace(C^{1/2}) < \infty$, as required by Lemma A.4.

\subsubsection*{Bounding \texttt{Term 5}}
Observe that
\begin{align*}
&\left\| (A+\lambda I)^{-1/2} \id^* (C-\hat C) (\id \beta_\lambda - \tbeta) \right\|_\calH \\
&=\| (\id^* C \id +\lambda I)^{-1/2} \id^* (C - \hat C) (\id \beta_\lambda - \tbeta) \|_\calH\\
&=\| (\id^* C \id +\lambda I)^{1/2}(\id^* C \id +\lambda I)^{-1} \id^* (C - \hat C) (\id \beta_\lambda - \tbeta) \|_\calH\\
&= \|(\id^* C \id +\lambda I)^{1/2}\id^* (CT+\lambda)^{-1} (C - \hat C) (\id\beta_\lambda -\tbeta)) \|_\calH\\
&= \left\| (\id^* C \id +\lambda I)^{1/2}\id^* (CT+\lambda I)^{-1} (C - \hat{C})\left(\id (\id^*C\id +\lambda I)^{-1} \id^* C \tbeta  - \tbeta\right)\right\|_\calH\\
&=\left\| (\id^* C \id +\lambda I)^{1/2}\id^* (CT+\lambda I)^{-1} (C - \hat{C}) \left((TC +\lambda I)^{-1}TC\tbeta - \tbeta \right) \right\|_\calH \\
&= \left\| (\id^* C \id +\lambda I)^{1/2}\id^* (CT+\lambda)^{-1}  (C -\hat{C}) (TC+\lambda I)^{-1}\left(TC\tbeta - (TC+\lambda I)\tbeta \right) \right\|_\calH \\
&= \lambda \left\| (\id^* C \id +\lambda I)^{1/2}\id^* (CT+\lambda I)^{-1} (C -\hat C) (TC+\lambda I)^{-1} \tbeta  \right\|_\calH  \\
&\le \lambda \left\| \id^* C \id +\lambda I\right\|^{1/2}\left\|\id^* (CT+\lambda I)^{-1} (C -\hat C) (TC+\lambda I)^{-1} \tbeta  \right\|  \\
&\le \lambda \left(\left\|T^{1/2}CT^{1/2}\right\|^{1/2}+\sqrt{\lambda}\right)\left\|  T^{1/2} (CT+\lambda)^{-1}  (C-\hat C) (TC+\lambda I)^{-1} \tbeta \right \| \\
&=\lambda  \left \| T^{1/2 -\alpha} T^\alpha (CT +\lambda I)^{-1}  (C-\hat C) (TC+\lambda I)^{-1} T^\alpha T^{-\alpha} \tbeta  \right\|\left(\left\|T^{1/2}CT^{1/2}\right\|^{1/2}+\sqrt{\lambda}\right)\\
&\leq  \lambda \| T\|^{1/2 -\alpha}  \|  T^\alpha (CT +\lambda I)^{-1}  (C-\hat C) (TC+\lambda I)^{-1} T^\alpha  \|  \| T^{-\alpha} \tbeta  \|\left(\left\|T^{1/2}CT^{1/2}\right\|^{1/2}+\sqrt{\lambda}\right),
 \end{align*}
 where $0<\alpha\le\frac{1}{2}$. Now, recalling the definition of $\Uptheta$ from~\eqref{eq:condition1}, and defining
\begin{align*}
 \hat\Uptheta &:=\frac{1}{n} \sum_{i=1}^ n T^\alpha (CT +\lambda I)^{-1}  (X_i\otimes X_i) (TC+\lambda I)^{-1} T^\alpha,
\end{align*}
we immediately have $\| T^\alpha (CT +\lambda I)^{-1}  (C-\hat C) (TC+\lambda I)^{-1} T^\alpha \|  = \| \hat\Uptheta- \Uptheta\| $. Hence, by Theorem A.3, 
we have with probability at least $1-e^{-\tau}$, 
\begin{align*}
 \| \hat\Uptheta- \Uptheta\|  \leq K_2 \| \Uptheta\| \frac{\sqrt{r(\Uptheta)} + \sqrt{\tau}}{\sqrt{n}},
\end{align*} 
for $\tau \geq 1$ and $n \geq (r(\Uptheta) \vee \tau)$, where $K_2$ is a universal constant. In other words, recalling the definition of $d(\lambda)$ from~\eqref{eq:condition1},  for $\delta \leq 1/e$, $n \geq d(\lambda)$, and $d(\lambda) \geq \ln(1/\delta)$, with probability at least $1-\delta$, we have 
\begin{align}\label{eq:term5condition}
&\| (A+\lambda I)^{-1/2} \id^* (C-\hat C) (\id \beta_\lambda - \tbeta) \|_\calH \nonumber\\  &\leq K_3 \left(\left\|T^{1/2}CT^{1/2}\right\|^{1/2}+\sqrt{\lambda}\right)   \|T\|^{1/2-\alpha} \|T^{-\alpha} \tbeta\|\lambda\| \Uptheta \|\sqrt{\frac{d(\lambda)}{n}}, 
\end{align}
for some universal constant $K_3$.

\subsubsection*{Bounding \texttt{Term 6}}
Note that
\begin{align}\label{eq:term6bound}
\| \id \beta_\lambda - \tbeta \|  =  \| (\id (\id^* C \id +\lambda I)^{-1} \id^* C \tbeta - \tbeta \| 
= \| T (CT+\lambda I)^{-1} C \tbeta - \tbeta \|.
\end{align}
The claim in Theorem~\ref{thm:masterthm}, immediately follows by combining~\eqref{eq:maindecomp},~\eqref{eq:term1bound}, \eqref{eq:term23bound},~\eqref{eq:term4bound},~\eqref{eq:term5condition} and~\eqref{eq:term6bound}.

\subsection{Proof of Theorem~\ref{thm:commutativeassumption}}\label{subsec:commutative1}
The proof of Theorem~\ref{thm:commutativeassumption} follows by carefully obtaining bounds on the individual terms in the inequality~\eqref{eq:masterbound} of the master theorem (Theorem~\ref{thm:masterthm}). 
Since $T$ and $C$ commute and have simple eigenvalues, they have the same eigenfunctions. 
Hence, recalling the definition of $\Xi$ in~\eqref{eq:condition2}, we have
\begin{align}
\| \Xi\| &=\left\Vert T (T^{1/2} C T^{1/2} + \lambda I)^{-2} T\right\Vert \nonumber\\
&= \sup_{i} \frac{\mu_i^2}{(\mu_i\xi_i+ \lambda)^2}\lesssim \sup_i \frac{i^{-2t}}{(i^{-(t+c) } +\lambda)^2}= \left[ \sup_i \frac{i^{-t}}{i^{-(t+c) } +\lambda)}\right]^2\nonumber\\
&\leq \left[ \lambda^{\tfrac{t -(t+c)}{t+c}} \right]^2 = \lambda^{-\tfrac{-2c}{t+c}},\label{eq:temp6}
\end{align}
where the last inequality follows from Lemma A.6. 
Next, recalling the definition of $N(\lambda)$ from~\eqref{eq:condition2} we have
\begin{align}
N(\lambda)&= \trace\left[(T^{1/2}CT^{1/2} +\lambda I)^{-1} T^{1/2} C T^{1/2} \right] \nonumber \\
&= \sum_{i} \frac{\mu_i\xi_i}{\mu_i\xi_i +\lambda} \lesssim \sum_{i} \frac{i^{-(t+c)}}{i^{-(t+c)} +\lambda}\lesssim \lambda^{-\tfrac{1}{t+c}}, \label{eq:temp7}
\end{align}
where the last inequality follows from Lemma A.5.
Next, recalling the definition of $\Uptheta$ in~\eqref{eq:condition1}, we have
\begin{align*}
\trace(\Uptheta)& = \trace\left({T^\alpha  (CT+\lambda I)^{-1} C (TC+\lambda I)^{-1} T^\alpha}\right)\\
&=\sum_i \frac{\mu_i^{2\alpha}\xi_i}{(\mu_i \xi_i +\lambda)^2} \lesssim \sum_i \frac{i^{-(2\alpha t +c)}}{(i^{-(t+c)} + \lambda)^2}\nonumber\\
&\lesssim \lambda^{-\tfrac{1+ (t+c)2 - (2\alpha t +c)}{t+c}} =  \lambda^{-\tfrac{1+ c + 2t (1-\alpha)}{t+c}},
\end{align*}
where the last inequality follows from Lemma A.5.
Next, we upper bound $\| \Uptheta\|$ as
\begin{align*}
\| \Uptheta \| &= \sup_i \frac{\mu_i^{2\alpha} ~\xi_i}{(\mu_i\xi_i + \lambda)^2} 
\lesssim  \sup_i \frac{i^{-(2\alpha t +c)}}{( i^{-(t+c)}+ \lambda)^2} = \left[ \sup_i \frac{i^{-(\alpha t +c/2)}}{i^{-(t+c)}+ \lambda}\right]^2  \\
&\leq \left[  \lambda^{\tfrac{\alpha t + c/2 - (t+c)}{t+c}}\right]^2  = \lambda^{\tfrac{2(\alpha -1)t - c}{t+c}},
\end{align*}
where the last inequality follows from Lemma A.6.
\textcolor{black}{We also bound $\| \Uptheta\|$ from below as}
\begin{align*}
\| \Uptheta \| = \sup_i \frac{\mu_i^{2\alpha}\xi_i}{(\mu_i \xi_i +\lambda)^2}  \geq  \sup_i \frac{\mu_i^{2\alpha}\xi_i}{(\| T^{1/2} C T^{1/2} \| +\lambda)^2}  =  \frac{ \| T^\alpha C T^\alpha\| }{(\| T^{1/2} C T^{1/2} \| +\lambda)^2}.
\end{align*}
Hence, we have
\begin{align}\label{eq:temp8}
\| \Uptheta \| \lesssim \lambda^{-\tfrac{2(1-\alpha)t + c}{t+c}}\quad \text{and}\quad \trace(\Uptheta) \lesssim  \lambda^{-\tfrac{1+ c + 2t (1-\alpha)}{t+c}},
\end{align}
and (recalling the definition of $d(\lambda)$ from~\eqref{eq:condition1})
\begin{align*}
d(\lambda) \leq \frac{\trace(\Uptheta)}{ \| T^\alpha C T^\alpha\|} (\| T^{1/2} C T^{1/2} \| +\lambda)^2 \lesssim 4 \| T^{1/2} C T^{1/2} \|^2 ~\lambda^{-\tfrac{1+c+2t(1-\alpha)}{t+c}}.
\end{align*}
Hence, the condition $n\gtrsim (d(\lambda) \vee \ln(1/\delta))$ is satisfied if $n \gtrsim \ln(1/\delta)$ and $n \gtrsim \lambda^{-\tfrac{1+c+2t(1-\alpha)}{t+c}} $ or equivalently $\lambda \gtrsim n^{-\tfrac{t+c}{1+c +2t(1-\alpha)}}$. Hence, the conditions on $n$ and $\lambda$ read as
\begin{align}\label{eq:temp9}
n \gtrsim \ln(1/\delta)~\quad\text{and}\quad n^{-\tfrac{t+c}{1+c +2t(1-\alpha)}} \lesssim \lambda \leq \| T^{1/2} C T^{1/2} \|. 
\end{align}
We now handle the bias term $\bias(\lambda)$. Denote by $(\psi_i)_{i\in\mathbb{N}}$, the eigenfunctions of the operator $T$. Recall that the assumption $\tbeta \in \range(T^\alpha)$, for $\alpha \in (0,1/2]$ implies that $\exists\, h \in \Ltwo$ such that $T^{\alpha} h =\tbeta$. Using this, we obtain
\begin{align}\label{eq:temp10}
&~\| T(CT + \lambda I)^{-1} C\tbeta -\tbeta \| = \| T(CT + \lambda I)^{-1} C T^\alpha h  - T^\alpha h  \| \\
=&~ \left[ \sum_i \left( \frac{\mu_i^{1+\alpha} \xi_i}{\mu_i\xi_i + \lambda} - \mu_i^\alpha \right)^2 \langle \psi_i, h \rangle^2 \right]^{1/2} 
=\left[ \sum_i \left( \frac{\lambda \mu_i^\alpha}{\mu_i\xi_i + \lambda} \right)^2 \langle \psi_i, h \rangle^2 \right]^{1/2}\nonumber \\
\leq&~ \lambda \sup_i \left[ \frac{\mu_i^\alpha}{\mu_i\xi_i +\lambda}\right] \|h\|\lesssim  \lambda \sup_i \left[ \frac{i^{-\alpha t}}{ i^{-(t+c)} + \lambda} \right] \| T^{-\alpha}\tbeta \|\nonumber \\
\lesssim&~ \lambda \lambda^{\tfrac{\alpha t - t- c}{t+c}}  \| T^{-\alpha}\tbeta \| =  \lambda^{\tfrac{\alpha t}{t+c}}  \| T^{-\alpha}\tbeta \|, \nonumber
\end{align}
where the last inequality follows from Lemma A.6. 
Combining~\eqref{eq:temp6}--\eqref{eq:temp10},
we obtain
\begin{align*}
&~\| \hat \beta - \tbeta \|\\
\lesssim&~ \lambda^{-\tfrac{2c}{4(t+c)}} \left[ \frac{\lambda^{-\tfrac{1}{2(t+c)}}}{\sqrt{n}} + \frac{\lambda(1+\sqrt{\lambda)} \lambda^{-\tfrac{2(1-\alpha)t+c}{2(t+c)}} \lambda^{-\tfrac{1+c+2t(1-\alpha)}{2(t+c)}}}{\sqrt{n}} \right] + \lambda^{\tfrac{\alpha t}{t+c}}\\
\lesssim&~ \lambda^{-\tfrac{2c}{4(t+c)}} \left[ \underbrace{\frac{\lambda^{-\tfrac{1}{2(t+c)}}}{\sqrt{n}}}_{p} + \underbrace{\frac{ \lambda^{-\tfrac{1+2t(1-2\alpha)}{2(t+c)}}}{\sqrt{n}}}_{q} \right] + \lambda^{\tfrac{\alpha t}{t+c}},
\end{align*} 
as $\sqrt{\lambda}=o(1)$ as $\lambda\rightarrow 0$.
Also $p = o(q)$ as $\lambda \to 0$. Therefore, we obtain 
\begin{align*}
\| \hat \beta - \tbeta \|& \lesssim \frac{\lambda^{-\tfrac{c}{2(t+c)}} \lambda^{-\tfrac{1+2t(1-2\alpha)}{2(t+c)}}} {\sqrt{n}} + \lambda^{\tfrac{\alpha t}{t+c}} = \frac{\lambda^{-\tfrac{1+c+2t(1-2\alpha)}{2(t+c)}}}{\sqrt{n}} +  \lambda^{\tfrac{\alpha t}{t+c}}.
\end{align*} 
Hence, by picking $\lambda$ as in~\eqref{eq:thm2claim} (which satisfies the condition on $\lambda$ in \eqref{eq:temp9}), we obtain the claim in~\eqref{eq:thm2claim}.

\subsection{Proof of Theorem~\ref{thm:tcteigendecay}}\label{subsec:noncommutative1}
We now prove Theorem~\ref{thm:tcteigendecay} by carefully upper bounding the terms in Theorem~\ref{thm:masterthm} under the assumptions of Theorem~\ref{thm:tcteigendecay}.
By recalling the definition of $\Xi$ from~\eqref{eq:condition2}, we have
\begin{align}\label{eq:temp11}
\| \Xi \|^{1/4} \le  \| T \|^{1/2} \left \| (T^{1/2} C T^{1/2} + \lambda I)^{-1} \right\|^{1/2} \lesssim 
\frac{1}{\sqrt{\lambda}}.
\end{align}
Next, recalling the definition of $N(\lambda)$ from~\eqref{eq:condition2}, we have
\begin{align}\label{eq:temp12}
N(\lambda) &= \trace\left[(T^{1/2}CT^{1/2} +\lambda I)^{-1} T^{1/2} C T^{1/2} \right]\nonumber \\
&=\sum_{i} \frac{\zeta_i}{\zeta_i +\lambda} \lesssim \sum_i \frac{i^{-b}}{i^{-b} + \lambda } \lesssim \lambda^{-\tfrac{1}{b}},
\end{align}
where the last inequality follows from Lemma A.5. 
Since $$\tbeta \in \range(T^{1/2} (T^{1/2} C T^{1/2})^\nu)\subset \range(T^{1/2}),$$ we have $\alpha=\frac{1}{2}$. 
Therefore, it follows from 
\eqref{eq:condition1} that
\begin{align}\label{eq:temp13}
\trace(\Uptheta)  & = \trace\left(  T^{1/2} (CT + \lambda I)^{-1} C (TC + \lambda I)^{-1} T^{1/2} \right) \nonumber \\
&\stackrel{(*)}{=}\trace\left(T^{1/2} T^{-1/2}  (\Lambda+\lambda I)^{-1} \Lambda  (\Lambda +\lambda I)^{-1} T^{-1/2} T^{1/2} \right)\nonumber\\
& = \trace\left(  T^{1/2} C T^{1/2} (T^{1/2} C T^{1/2} + \lambda )^{-2}\right) \nonumber \\
& = \sum_i \frac{\zeta_i}{(\zeta_i +\lambda )^2} \lesssim\sum_i \frac{i^{-b}}{ (i^{-b} + \lambda)^2} 
 \lesssim \lambda^{-\tfrac{1+b}{b}},
\end{align}
where the last inequality follows from Lemma A.5 
and $\Lambda:=T^{1/2}CT^{1/2}$ in $(*)$. Furthermore, we have the following upper bound on $\| \Uptheta \|$ as 
\begin{align}\label{eq:temp14}
\|\Uptheta \| & = \left\| T^{1/2} (CT+\lambda I)^{-1} C (TC+\lambda I)^{-1} T^{1/2}  \right\| \nonumber \\
& = \left\| (T^{1/2}C T^{1/2} +\lambda I)^{-1} (T^{1/2}C T^{1/2}) (T^{1/2}C T^{1/2} +\lambda I)^{-1} \right\|  \nonumber\\
& =\sup_i  \frac{\zeta_i}{(\zeta_i + \lambda)^2} 
 \lesssim \sup_i\ \frac{i^{-b}}{ (i^{-b} + \lambda)^2}  = \left[\sup_i \frac{i^{-b/2}}{i^{-b} + \lambda} \right]^2  \lesssim  \frac{1}{\lambda},
\end{align}
where the last inequality follows from Lemma A.6. 
We also have the following \textcolor{black}{lower bound} on $\| \Uptheta \|$ as
\begin{align}\label{eq:temp15}
\| \Uptheta \| = \sup_i \frac{\zeta_i}{(\zeta_i + \lambda)^2} \geq \frac{\| T^{1/2} C T^{1/2} \| }{(\| T^{1/2} C T^{1/2}\| + \lambda)^2}.
\end{align}
Hence, recalling the definition of $d(\lambda)$ from~\eqref{eq:condition1} and the fact that $\lambda \leq \| T^{1/2} C T^{1/2} \|$, we have
\begin{align*}
d(\lambda) \leq  \frac{\lambda^{-\tfrac{(1+b)}{b}}}{\| T^{1/2} C T^{1/2}\|} \left( \|  T^{1/2} C T^{1/2} \| + \lambda  \right)^2 \leq   4 \| T^{1/2} C T^{1/2}\| \lambda^{-\tfrac{1+b}{b}} \lesssim \lambda^{-\tfrac{1+b}{b}}.
\end{align*}
Therefore, the condition $ n \gtrsim (d(\lambda) \vee \ln(1/\delta))$ is satisfied if $n \gtrsim \ln(1/\delta)$ and $n \gtrsim \lambda^{-\tfrac{1+b}{b}}$ or equivalently $\lambda\gtrsim n^{-\tfrac{b}{1+b}}$. Hence, the conditions on $n$ and $\lambda$ read as
\begin{align*}
n \gtrsim \ln(1/\delta) \qquad\text{and}\qquad n^{-\tfrac{b}{1+b}}\lesssim \lambda\le \|T^{1/2} C T^{1/2} \|.
\end{align*}
Finally to handle the bias term, the assumption $\tbeta \in \range( T^{1/2} (T^{1/2} C T^{1/2})^\nu)$, for $\nu \in (0,1]$ implies that $\exists\, h \in \Ltwo$ such that $T^{1/2} (T^{1/2} C T^{1/2})^\nu h \in \tbeta$. Therefore, we have
\begin{align}\label{eq:temp16}
&~\| T(CT + \lambda I)^{-1} C\tbeta -\tbeta \| \nonumber \\
  \stackrel{(*)}{=} &~ \|  T^{1/2} (\Lambda + \lambda I)^{-1} \Lambda  \Lambda ^\nu h - T^{1/2}  \Lambda ^\nu h\| \nonumber \\
\leq &~\| T\|^{1/2} \left \| (T^{1/2}CT^{1/2} + \lambda I)^{-1} (T^{1/2}CT^{1/2})^{\nu+1} h -  ( T^{1/2} C T^{1/2} )^\nu h \right\| \nonumber \\
\leq &~ \| T\|^{1/2} \sup_i \left| \frac{\zeta_i^{1+\nu}}{\zeta_i +\lambda} - \zeta_i^\nu \right| \| h\| \lesssim \lambda \sup_i\left[ \frac{i^{-b\nu}}{i^{-b} +\lambda}\right] \lesssim \lambda \lambda^{\tfrac{b\nu -b}{b}}  =\lambda^\nu,
\end{align}
where $\Lambda:=T^{1/2}CT^{1/2}$ in $(*)$. By combining the bounds in~\eqref{eq:temp11}--\eqref{eq:temp16}, we obtain
\begin{align*}
\| \hat \beta - \tbeta\| \lesssim \frac{1}{\sqrt{\lambda}} \left[\frac{\lambda^{-\tfrac{1}{2b}}}{\sqrt{n}} + \frac{\lambda(1+\sqrt{\lambda}) \lambda^{-1/2} \lambda^{-\tfrac{1+b}{2b}}}{\sqrt{n}}\right] + \lambda^\nu \leq  \frac{\lambda^{-\left( \tfrac{1}{2} + \tfrac{1}{2b} \right)}}{\sqrt{n}} + \lambda^\nu. 
\end{align*}
Thus, by setting $\lambda$ as in~\eqref{eq:thm3claim}, we obtain the claim in~\eqref{eq:thm3claim}. 

\subsection{Proof of Theorem~\ref{thm:alignedeigensystem}}\label{subsec:eigalign}
We prove Theorem~\ref{thm:alignedeigensystem} by obtaining a better bound on $\bias(\lambda)$ under the assumptions in Theorem~\ref{thm:alignedeigensystem}. Define $\Lambda:=T^{1/2}CT^{1/2}$.  
The assumption that $\tbeta \in \range(T^{1/2} \Lambda^\nu)$ implies that $\exists\, h \in \Ltwo$ such that $T^{1/2} \Lambda^\nu h  =\tbeta$. Hence, we have
\begin{align*}
&\bias(\lambda)=\| T(CT+\lambda I)^{-1}C\tbeta -\tbeta \|= \left\| T^{1/2} (\Lambda+\lambda I)^{-1} \Lambda^{1+\nu} h-T^{1/2} \Lambda^\nu h\right\| \\
&= \left\|  T^{1/2} \sum_i \frac{\zeta_i^{1+\nu}}{\zeta_i + \lambda} \langle \phi_i, h\rangle \phi_i  - T^{1/2} \sum_i \zeta_i^\nu \langle \phi_i, h\rangle \phi_i \right \|
= \left\|  \sum_i\left(\frac{\zeta_i^{1+\nu}}{\zeta_i + \lambda} - \zeta_i^\nu \right) \langle \phi_i, h\rangle T^{1/2} \phi_i \right \|\\
&= \lambda \left\| \sum_i  \frac{\zeta_i^{\nu}}{\zeta_i + \lambda}\langle \phi_i, h\rangle T^{1/2} \phi_i  \right\| 
= \lambda \left\|  \sum_i \frac{\zeta_i^\nu}{\zeta_i+\lambda} \langle \phi_i, h \rangle \sum_j \sqrt{\mu_j} \langle \phi_i,\psi_j\rangle \psi_j \right\|\\
&= \lambda \left\|  \sum_j \sqrt{\mu_j} \left[\sum_i \frac{\zeta_i^\nu}{\zeta_i+\lambda} \langle \phi_i, h \rangle \langle \phi_i,\psi_j\rangle \right] \psi_j \right\|
= \lambda \left\{ \sum_j\mu_j \left[\sum_i \frac{\zeta_i^\nu}{\zeta_i+\lambda} \langle \phi_i, h \rangle \langle \phi_i,\psi_j\rangle \right]^2 \right\}^{1/2} \\
&= \lambda \left\{  \sum_i\sum_\ell \sum_j \mu_j \langle \phi_i,\psi_j \rangle \langle \phi_\ell,\psi_j \rangle \langle \phi_i,h \rangle \langle \phi_\ell,h \rangle \frac{\zeta_i^\nu}{\zeta_i+\lambda} \frac{\zeta_\ell^\nu}{\zeta_\ell+\lambda}\right\}^{1/2}\\
&= \lambda \left\{  \sum_i\sum_\ell \left[\frac{\zeta_i^\nu\zeta_\ell^\nu} {(\zeta_i+\lambda)(\zeta_\ell+\lambda)}  \sum_j \mu_j  \langle \phi_i,\psi_j \rangle \langle \phi_\ell,\psi_j \rangle \right] \langle \phi_i,h \rangle \langle \phi_\ell,h \rangle  \right\}^{1/2}\\
&\leq  \lambda \left\{\left\{\sum_i\sum_\ell \left[\frac{\zeta_i^\nu\zeta_\ell^\nu} {(\zeta_i+\lambda)(\zeta_\ell+\lambda)}  \right]^2 \left[ \sum_j \mu_j  \langle \phi_i,\psi_j \rangle \langle \phi_\ell,\psi_j \rangle \right]^2 \right\}^{1/2}\left\{ \sum_i \sum_\ell \langle \phi_i,h \rangle^2 \langle \phi_\ell,h \rangle^2\right\}^{1/2}   \right\}^{1/2}\\
&\leq  \lambda \|h\| \left\{  \sum_{i,\ell} \left[\frac{\zeta_i^\nu\zeta_\ell^\nu} {(\zeta_i+\lambda)(\zeta_\ell+\lambda)}  \right]^2 \left[ \sum_j \mu_j  \langle \phi_i,\psi_j \rangle \langle \phi_\ell,\psi_j \rangle \right]^2   \right\}^{1/4} \\
&= \lambda \| h \| \left\{  \sum_{i,\ell} \left[\frac{\zeta_i^\nu\zeta_\ell^\nu \sqrt{\mu_i \mu_\ell}} {(\zeta_i+\lambda)(\zeta_\ell+\lambda)}  \right]^2   \left[ \sum_j \frac{\mu_j}{\sqrt{\mu_i \mu_\ell}}  \langle \phi_i,\psi_j \rangle \langle \phi_\ell,\psi_j \rangle \right]^2  \right\}^{1/4}\\
&\le \lambda \| h \|  \left[  \sum_{i,\ell} \left[\frac{\zeta_i^\nu\zeta_\ell^\nu \sqrt{\mu_i \mu_\ell}} {(\zeta_i+\lambda)(\zeta_\ell+\lambda)}  \right]^2   \right]^{1/4} \left\{ \sup_{i,l}\left[ \sum_j \frac{\mu_j}{\sqrt{\mu_i \mu_\ell}}  \langle \phi_i,\psi_j \rangle \langle \phi_\ell,\psi_j \rangle \right]^2  \right\}^{1/4}\\
&= \lambda \|h\|  \left[  \sum_{i} \left(\frac{\zeta_i^\nu  \sqrt{\mu_i}} {\zeta_i+\lambda}\right)^2 \right]^{1/2} \left\{ \sup_{i,\ell} \frac{1}{\mu_i\mu_\ell} \left| \sum_j \mu_j \langle \phi_i, \psi_j \rangle \langle \phi_\ell, \psi_j\rangle \right|^2 \right\}^{1/4} \\
&\stackrel{(*)}{\lesssim} \lambda \|h\| \left[  \sum_{i} \left(\frac{\zeta_i^\nu  \sqrt{\mu_i}} {\zeta_i+\lambda}\right)^2 \right]^{1/2}\lesssim\lambda\Vert h\Vert \left[\sum_i \frac{i^{-(2b\nu+t)}}{(i^{-b}+\lambda)^2}\right]^{1/2}\\
&\stackrel{(**)}{\lesssim} \lambda\cdot \lambda^{-(1+2b-2b\nu-t)/2b} \Vert h\Vert,
\end{align*}
where $(*)$ follows from the assumption in \eqref{eq:boundedDcondition} and $(**)$ from Lemma A.5 
when $b \geq 2b\nu +t$ and $2b\geq 2b\nu+t$, i.e., $\nu \leq \frac{1}{2} - \frac{t}{2b}$. Therefore,
\begin{align}\label{eq:temp17}
\bias(\lambda) = \| T(CT+\lambda)^{-1} C\tbeta - \tbeta\|  \lesssim \lambda^{\tfrac{2b\nu+t-1}{2b}}   ~\|h\|.
\end{align} 
Combining the bounds in \eqref{eq:temp11}--\eqref{eq:temp15} from the proof of Theorem~\ref{thm:tcteigendecay} and~\eqref{eq:temp17}, we obtain
\begin{align*}
\| \hat\beta - \tbeta \| \lesssim \frac{\lambda^{-\left( \frac{1}{2} + \frac{1}{2b}   \right)}}{\sqrt{n}} +\underbrace{\lambda^{\tfrac{2b\nu+t-1}{2b}}}_{p}.
\end{align*}
Note that we have $p = o(\lambda^\nu)$ as $\lambda \to 0$. Hence, by our choice of $\lambda$, the result follows.
\subsection{Proof of Proposition~\ref{prop:char}}\label{subsec:range}
First note that $k$ and $c$ are positive definite kernels which follow from the form of $k$ and $c$ and the assumption that $a_i\ge 0$ for all $i$ and $b_m\ge 0$ for all $m$. Next, we note that we $k(\cdot, x) = \sum_{i} a_i \phi_i(x) \phi_i(\cdot) \in \calH$, and
\begin{align*}
    \langle f, k(\cdot, x)\rangle_{\calH} = \sum_i \frac{f_i a_i}{a_i} \phi_i(x) = f(x),~~~ \forall x \in [0,1],
\end{align*}
implying $\calH$ is an RKHS induced by the kernel $k$.

By considering the integral operator $Tf = \int_0^1 k(\cdot, x) f(x) dx$, for $f \in L^2([0,1])$, we have
\begin{align*}
    Tf = \int_0^1 \sum_i a_i \phi_i(x) f(x) \phi_i dx=\sum_i a_i \left[ \int_0^1 \phi_i(x) f(x) d(x) \right] \phi_i
    =\sum_i a_i \langle f, \phi_i\rangle \phi_i.
\end{align*}
This implies $T \phi_j = \sum_i a_i \langle \phi_j, \phi_i \rangle \phi_i= \sum_i a_i \delta_{ij} \phi_i = a_j \phi_j$, i.e., $(a_i, \phi_i)_{i \in \mathbb{N}}$ form the pair of eigenvalues and eigenfunctions  of the operator $T$. Since  $\sum_i a_i < \infty$ and $a_i \geq 0$, $T$ is also a positive, trace-class operator.

Since $X$ is a mean-zero Gaussian process with covariance function $$c(x,y) = \sum_m b_m \psi_m(x) \psi_m(y),$$ it follows from the Karhunen-Lo\'{e}ve theorem that $X$ has a representation of the form $$X=\sum_m \sqrt{b_m} z_m \psi_m,$$ where $z_m \sim N(0,1)$. Furthermore, we have 
\begin{align*}
    C&=\E[X \otimes X] = \E\left[\sum_{m,s} \sqrt{b_m}\sqrt{b_s} z_m z_s \psi_m\otimes \psi_s\right] 
    =\sum_{m,s} \sqrt{b_m}\sqrt{b_s} \E[z_mz_s]\psi_m\otimes \psi_s\\
     &=\sum_{m,s} \sqrt{b_m}\sqrt{b_s} \delta_{ms}\psi_m\otimes \psi_s=\sum_m b_m \psi_m\otimes \psi_m.
\end{align*}
This implies that $C\psi_\ell=\sum_m b_m \psi_m \langle \psi_m, \psi_\ell\rangle = b_\ell\psi_\ell$. Since $b_m \geq 0$ and $\sum_m b_m < \infty$, $C$ is also a positive, trace-class operator with $(b_i, \psi_i)_{i \in \mathbb{N}}$ as eigenvalue-eigenfunction pairs. 

We next characterize the space $\mathscr{R}(T^{1/2} (T^{1/2}CT^{1/2}))$. Note that by functional calculus, we have $T^{1/2} = \sum_i \sqrt{a_i} \phi_i\otimes \phi_i$. Hence,
\begin{align*}
    CT^{1/2} =\sum_{m,i} b_m\sqrt{a_i} \langle \psi_m,\phi_i\rangle \psi_m\otimes \phi_i = \sum_{m,i} b_m\sqrt{a_i} \theta_{mi} \psi_m\otimes \phi_i
\end{align*}
and
\begin{align}
    T^{1/2}CT^{1/2} &= \left[ \sum_j \sqrt{a_j} \phi_j \otimes \phi_j \right]\left[ \sum_{m,i} b_m\sqrt{a_i} \theta_{mi} \psi_m\otimes \phi_i\right]\nonumber\\
    &=\sum_{i,j,m} \sqrt{a_i a_j} b_m \theta_{mi} \theta_{mj} \phi_j \otimes \phi_i\nonumber\\
    &=\sum_{i,j}\sqrt{a_i a_j} \left[ \sum_m b_m \theta_{mi} \theta_{mj} \right] \phi_j \otimes \phi_i\label{eq:tctoperator}\\
    &=\sum_{ij}\sqrt{a_i a_j} \eta_{ij} \phi_j \otimes \phi_i.\nonumber
\end{align}
By a similar calculation, we obtain 
\begin{align*}
    T^{1/2}(T^{1/2}CT^{1/2}) = \sum_{ij} a_j \sqrt{a_i}\eta_{ij} \phi_j\otimes \phi_i.
\end{align*}
Since
\begin{align*}
\mathscr{R}(T^{1/2} (T^{1/2}CT^{1/2}))
= \left\{ \tilde{f} \in L^2([0,1])|\tilde{f}= T^{1/2} (T^{1/2}CT^{1/2}) h: h \in L^2([0,1]) \right\},
\end{align*}
for any function $\tilde{f} \in \mathscr{R}(T^{1/2} (T^{1/2}CT^{1/2}))$, we have $\tilde{f} = T^{1/2} (T^{1/2}CT^{1/2}) h$ for some $h\in L^2([0,1])$. By defining $h_i\coloneqq\langle h,\phi_i\rangle$ and $\beta_j\coloneqq \sum_i \sqrt{a_i} \eta_{ij}h_i$, we have
\begin{align*}
    \tilde{f} =\sum_{ij} a_j \sqrt{a_i} \eta_{ij} h_i \phi_j = \sum_ia_j\beta_j\phi_j.
\end{align*}
We now show that $\tilde{f} \in \mathcal{H}$. Consider $\|\tilde{f}\|^2_{\calH}=\sum_j a_j \beta_j^2 \leq \|h\|^2 \sum_{j} a_j \sum_i a_i\eta_{ij}^2$, where
\begin{align*}
\eta_{ij}^2& = \left(\sum_m b_m \theta_{mi} \theta_{mj}\right)^2\leq \sum_m b_m^2 \theta^2_{mi} \sum_m \theta^2_{mj} = \|\phi_j\|^2 \sum_m b_m^2 \theta^2_{mi} \\
&=\sum_m b_m^2 \theta^2_{mi} \leq \sum_m b_m^2 \leq \sum_m b_m.
\end{align*}
Hence, $\|\tilde{f}\|^2_{\calH} \leq \|h\|^2 \left(\sum_m b_m\right) (\sum_i a_i)^2 < \infty$, which implies that we have $\mathscr{R}(T^{1/2} (T^{1/2}CT^{1/2})) \subset \calH$.

Next, we show that $\mathscr{R}(T^{1/2} (T^{1/2}CT^{1/2})) \subset \tilde{\mathcal{H}}$. To this end, note that for any function $\tilde{f} \in \mathscr{R}(T^{1/2} (T^{1/2}CT^{1/2}))$, we have that $\tilde{f} = \sum_j a_j \beta_j \phi_j$, and by a similar calculation as above we have that  $\|\tilde{f}\|^2_{\tilde{\mathcal{H}}} < \infty$ where we used $\beta^2_j\le \Vert h\Vert^2 \tau_j$. Finally, since 
\[
\sum_i \frac{f_i^2}{a_i} = \sum_i \frac{f_i^2\tau_i}{a_i\tau_i} \leq \sup_i \tau_i \sum_i \frac{f_i^2}{a_i \tau_i} < \infty.
\]
it also follows that $\tilde{\mathcal{H}} \subset \mathcal{H}$.

\subsection{Proof of Theorem~\ref{thm:predictionmasterthm}}\label{subsec:main-pred}
Note that
\begin{align*}
\| C^{1/2} (\hat \beta -\beta^*)\| 
&\leq \|C^{1/2}(\id\hat \beta -\id\beta_\lambda)\|+\|C^{1/2} (\id\beta_\lambda-\beta^*)\|\\
&= \| (\id^*C\id)^{1/2} (\hat\beta - \beta_\lambda) \|_\calH +  \| C^{1/2} (\id \beta_\lambda -\beta^*) \|.
\end{align*}
By defining $A:=\id^*C\id$ and following the steps similar to the proof of Theorem~\ref{thm:masterthm} in bounding the first term, we  obtain
\begin{align*}
\| C^{1/2} (\hat \beta -\tbeta)\|^2 &\lesssim   \underbrace{\| A^{1/2}(A + \lambda I)^{-1/2} \|^2}_{\texttt{Term 7}}\cdot \left( \texttt{Term 2} \right)^2\cdot \left(  \texttt{Term 3} \right)^2\\  &\qquad\cdot \left[ \texttt{Term 4}  +  \texttt{Term 5} \right]^2 + \underbrace{ \| C^{1/2} (\id\beta_\lambda -\beta^*)\|^2}_{\texttt{Term 8}}. 
\end{align*}

We will now proceed with bounding \texttt{Term 7} and \texttt{Term 8}.
For \texttt{Term 7}, note that 
\begin{align*}
\| A^{1/2}(A + \lambda I)^{-1/2} \| \leq 1.
\end{align*}
To bound \texttt{Term 8}, note that 
\begin{align*}
\|C^{1/2} (\id \beta_\lambda - \beta^*) \| & = \| C^{1/2} \id (\id^* C \id +\lambda I)^{-1} \id^* C \beta^* - C^{1/2} \beta^*\| \\
&= \| C^{1/2} T (CT+\lambda I)^{-1} C\beta^* - C^{1/2} \beta^* \|.
\end{align*}
The result therefore follows by combining the bounds on \texttt{Term 7} and \texttt{Term 8}, along with the bounds for \texttt{Term 2} to \texttt{Term 5} from the proof of Theorem~\ref{thm:masterthm}.

\subsection{Proof of Theorem~\ref{thm:predictionnoncommutativeassumption}}\label{subsec:commutative2}
We first deal with the bias term. Since $\beta^* \in \range(T^{1/2})$, $\exists\, h \in \Ltwo$ such that $\beta^* = T^{1/2} h$.  Therefore, we have
\begin{align*}
&~~~~\| C^{1/2} T (CT +\lambda I)^{-1} C\beta^* - C^{1/2}\beta^* \|  \\ & = \| C^{1/2} T^{1/2} (T^{1/2}CT^{1/2} +\lambda I)^{-1} T^{1/2} C T^{1/2} h - C^{1/2} T^{1/2} h  \| \\
& = \left\| C^{1/2} T^{1/2} \left[ (T^{1/2} C T^{1/2} +\lambda I)^{-1} T^{1/2} C T^{1/2} h - h \right] \right\|\\
& = \|(T^{1/2} C T^{1/2})^{1/2} (T^{1/2} C T^{1/2} + \lambda I)^{-1} T^{1/2} C T^{1/2} h - (T^{1/2} C T^{1/2})^{1/2} h \|\\
& = \left[  \sum_i \left( \frac{\zeta_i^{3/2}}{\zeta_i +\lambda} - \zeta_i^{1/2} \right)^2 \langle \phi_i, h \rangle^2 \right]^{1/2} \leq \lambda\|h \| \sup_i \frac{\zeta^{1/2}_i}{\zeta_i +\lambda}\\
&\lesssim \lambda \|h \| \sup_i \frac{i^{-\tfrac{b}{2}}}{i^{-b} + \lambda } \lesssim \lambda \| h \| \lambda^{-1/2} = \sqrt{\lambda} \|h \|,
\end{align*}
where the last inequality follows from Lemma A.6. 
Since $\alpha =1/2$, by using~\eqref{eq:temp12}, \eqref{eq:temp13} and~\eqref{eq:temp14} respectively for bounding $N(\lambda)$, $\trace(\Uptheta)$ and $\| \Uptheta\|$, along with the above bound on the bias, we obtain
\begin{align*}
\| C^{1/2} (\hat \beta -\beta^*) \|^2 \lesssim \frac{\lambda^{-\tfrac{1}{b}}}{n} +\lambda.
\end{align*}
Hence, by setting $\lambda$ as in~\eqref{thm7claim}, we obtain the claim in~\eqref{thm7claim}. 

\subsection{Proof of Theorem~\ref{thm:predictioncommutativeassumption}}\label{subsec:noncommutative2}
We first bound the bias term. Since $\beta^* \in \range(T^\alpha)$, $\exists\, h \in \Ltwo$ such that $\beta^* = T^\alpha h$. Therefore, we have
\begin{align*}
&\|C^{1/2} T (CT +\lambda I)^{-1} C\beta^* - C^{1/2}\beta^* \|\\
&=\| C^{1/2} T (CT +\lambda I)^{-1} C T^\alpha h - C^{1/2} T^\alpha h \| \\
&= \| C^{1/2} T^{1/2} (T^{1/2}CT^{1/2} +\lambda I)^{-1} T^{1/2} C T^\alpha h - C^{1/2} T^\alpha h \| \\
&= \left[\sum_i \left(  \frac{\mu_i^{1+\alpha} \xi_i^{3/2}}{\mu_i\xi_i + \lambda}  -\xi^{1/2}_i \mu_i^\alpha  \right)^2 \langle\phi_i, h\rangle^2   \right]^{1/2} \\
&\leq \lambda \| h\| \sup_i \frac{\xi_i^{1/2} \mu_i^\alpha}{\mu_i \xi_i +\lambda}\lesssim \lambda \| h \| \sup_i \frac{i^{-(\alpha t +c/2)}}{i^{-(t+c)} +\lambda}\\
&\lesssim \lambda \|h\| \lambda^{\tfrac{\alpha t + c/2 -t -c}{t+c}} = \lambda^{\tfrac{\alpha t +c/2}{t+c}} \| h \|,
\end{align*}
where the last inequality follows from Lemma A.6. 
This upper bound on the bias, along with~\eqref{eq:temp7} and~\eqref{eq:temp8} from the proof of Theorem~\ref{thm:commutativeassumption} implies that
\begin{align*}
\| C^{1/2} (\hat \beta - \beta^*) \|^2 \lesssim \frac{\lambda^{-\tfrac{1+2t(1-2\alpha)}{t+c}}}{n} + \lambda^{\tfrac{2\alpha t+c}{t+c}}.
\end{align*} 
Thus by setting $\lambda$ as in~\eqref{eq:thm2claim}, we obtain the claim in~\eqref{eq:thm6claim}. 

\section*{Acknowledgements} 
KB is supported in part by the National Science Foundation (NSF) grant DMS-2053918 and a UC Davis Center for Data Science and Artificial Intelligence Research (CEDAR) seed grant. HGM is supported in part by NSF grant DMS-2014526 and NIH grant UG3-0D023313. BKS is supported in part by the NSF CAREER Award DMS-1945396.

\references

\appendix

\section{Auxiliary results} \label{sec:auxappendix}

In this section, we collect some technical results used to prove the main results. 

\begin{theorem}[\citealp{shih2011stein, kuo2011integration}] Let $H$ be a separable Hilbert space, with norm $\|\cdot\|_H$ and inner-product $\langle \cdot, \cdot \rangle_H$. Let $\tilde{H}$ be the completion with respect to $\|\cdot\|_H$. Let $p$ be a Gaussian measure on $\tilde{H}$. Then, for any $h\in H$ and for any once Fr\'echet-differentiable function $f:\tilde{H} \to \mathbb{R}$, we have $\int_{\tilde{H}} \langle h,x\rangle_H \, dp(x) = \int_{\tilde{H}} \langle\nabla f(x),h\rangle_H\, dp(x)$, where $\nabla$ represents the Fr\'echet derivative, as long as the expectations are well-defined.
\end{theorem}

\begin{lemma}[Chebychev's inequality for Hilbert-valued random variables]\label{lem:chebychev}
Let $Z_i\in H$, for $i=1, \ldots, n$ be i.i.d.~Hilbert-valued random variables such that $\E[Z_i] =0$. Then 
\begin{align*}
\mathbb{P}\left( \left\| \frac{1}{n} \sum_{i=1}^n Z_i \right\|_H \geq \sqrt{\frac{\E\| Z_1\|_\calH^2}{n\delta}}\right) \leq \delta.
\end{align*} 
\end{lemma}
\begin{proof}
By Markov's inequality, it is obvious that for any $\epsilon>0$  $$\mathbb{P}\left( \left\| \frac{1}{n} \sum_{i=1}^n Z_i \right\|_H \geq \epsilon\right) \leq \frac{\E\left\Vert\frac{1}{n}\sum^n_{i=1}Z_i\right\Vert^2_H}{\epsilon^2}.$$
By noting $$\E\left\Vert\frac{1}{n}\sum^n_{i=1}Z_i\right\Vert^2_H=\frac{1}{n^2}\sum^n_{i,j=1}\E\langle Z_i,Z_j\rangle_H=\frac{1}{n^2}\sum^n_{i=1}\E\left\Vert Z_i\right\Vert^2_H+\frac{1}{n^2}\sum^n_{i\ne j}\E\langle Z_i,Z_j\rangle_H=\frac{\E\| Z_1\|_H^2}{n}$$ and choosing $\epsilon=\sqrt{\frac{\E\| Z_1\|_H^2}{n\delta}}$ yields the result.
\end{proof}
\begin{theorem}[\citealp{Koltchinskii-17}]\label{Thm:kL}
Let $X_1,\ldots,X_n$ be i.i.d.~centered Gaussian random variables in a separable Hilbert space $H$ with covariance operator $\Sigma=\mathbb{E}[X\otimes_H X]$. Let $\hat{\Sigma}=\frac{1}{n}\sum_{i=1}X_i\otimes_H X_i$ be the empirical covariance operator. Define $$r(\Sigma):= \frac{\left(\mathbb{E}\Vert X\Vert_{H}\right)^2}{\Vert \Sigma\Vert_{\emph{op}}}.$$ Then for all $\tau\ge 1$ and $n\ge (r(\Sigma)\vee \tau)$, with probability at least $1-e^{-\tau}$,
$$\Vert \hat{\Sigma}-\Sigma\Vert_\emph{op}\le K_1\Vert \Sigma\Vert_{\emph{op}}\frac{\sqrt{r(\Sigma)}+\sqrt{\tau}}{\sqrt{n}},$$
where $K_1$ is a universal constant independent of $\Sigma$, $\tau$ and $n$.
\end{theorem}
\begin{lemma}\label{lem:term4relatedlemma}
For any bounded, self-adjoint positive operator $\Gamma$ on $\Ltwo$, 
$$ \E [\langle X, \Gamma X \rangle^2] \leq 3 \emph{\trace}^2(\Gamma C),$$
assuming $\emph{\trace}(C^{1/2}) < \infty$ with 
$C = \E [X \otimes X]$.
\end{lemma}
\begin{proof}
First note that by the Karhunen-Lo\'eve expansion, we have $X = \sum_i x_i \varphi_i$, where $(\varphi_i)_{i\in\mathbb{N}}$ and $(\lambda_i)_{i\in\mathbb{N}}$ are the eigenfunctions and eigenvalues of $C$, and $x_i$ are independent Gaussian random variables with $\E[x_i^2] =\lambda_i$. Hence, we have
\begin{align}
\E[\langle X, \Gamma X \rangle^2] &= \E \left[ \left\langle  \sum_i x_i \varphi_i, \sum_j x_j \Gamma \varphi_j  \right\rangle^2 \right] = \E \left[\sum_{i,j} x_i x_j \langle \varphi_i, \Gamma\varphi_j \rangle   \right]^2\nonumber\\
& \stackbin[]{(\ast)}{=}\sum_{i,j,k,\ell} \E [x_i x_j x_k x_\ell] \langle \varphi_i, \Gamma \varphi_j \rangle \langle \varphi_k, \Gamma \varphi_\ell \rangle \nonumber\\
&= \sum_i \E[x_i^4] \langle \varphi_i, \Gamma \varphi_i \rangle^2 +  \sum_{i\neq j} \E[x_i^2] \E[x_j^2] \langle \varphi_i, \Gamma \varphi_i \rangle \langle \varphi_j, \Gamma \varphi_j \rangle\nonumber\\
&= 3 \sum_i \lambda_i^2  \langle \varphi_i, \Gamma \varphi_i \rangle^2 +  \left( \sum_i \lambda_i  \langle \varphi_i, \Gamma \varphi_i \rangle\right)^2 - \sum_i \lambda_i^2  \langle \varphi_i, \Gamma \varphi_i \rangle^2\nonumber\\
& = 2 \sum_i \lambda_i^2  \langle \varphi_i, \Gamma \varphi_i \rangle^2 + \left( \sum_i \lambda_i  \langle \varphi_i, \Gamma \varphi_i \rangle\right)^2, \label{Eq:tempo}
\end{align}
where the exchange of expectation and summation in $(*)$ holds by Fubini's theorem if \begin{align}
\sum_{i,j,k,\ell} \E |x_ix_jx_k x_\ell| |\langle \varphi_i, \Gamma \varphi_j \rangle| |\langle \varphi_k, \Gamma \varphi_\ell \rangle| < \infty. \label{Eq:fubini}
\end{align}
We will later verify \eqref{Eq:fubini}. Note that
\begin{align}
\trace(\Gamma C)&  = \trace\left(\Gamma  \left( \sum_i \lambda_i \varphi_i \otimes \varphi_i \right) \right) \nonumber \\
&= \trace\left(\sum_i \lambda_i \left( \Gamma \varphi_i\right) \otimes \varphi_i \right)=\sum_i\lambda_i\langle \varphi_i,\Gamma\varphi_i\rangle.\label{Eq:temp1}
\end{align}
Furthermore, recall that the Hilbert-Schmidt norm of an operator $A$ is defined as $\|A\|^2_{HS}\coloneqq \sum_i\| Ae_i \|^2$, where $(e_i)_{i\in\mathbb{N}}$, is any orthonormal basis for $\Ltwo$. Hence, we have 
\begin{align*}
&\left\| \sum_i \lambda_i  \langle \varphi_i, \Gamma \varphi_i \rangle \left( \varphi_i \otimes \varphi_i \right) \right\|_{HS}^2\\
&=\left\langle  \sum_i \lambda_i  \langle \varphi_i, \Gamma \varphi_i \rangle \left( \varphi_i \otimes \varphi_i\right), \sum_j \lambda_j  \langle \varphi_j, \Gamma \varphi_j \rangle \left( \varphi_j \otimes \varphi_j\right) \right\rangle_{HS}\\
 &= \sum_{i,j} \lambda_i \lambda_j \langle \varphi_i, \Gamma \varphi_i\rangle \langle \varphi_j, \Gamma \varphi_j\rangle  \left\langle \varphi_i\otimes \varphi_i, \varphi_j \otimes \varphi_j \right\rangle_{HS}\\
&=  \sum_{i,j} \lambda_i \lambda_j \langle \varphi_i, \Gamma \varphi_i\rangle \langle \varphi_j, \Gamma \varphi_j\rangle  \left\langle \varphi_i, \varphi_j \right\rangle^2
 = \sum_i \lambda_i^2 \langle \varphi_i, \Gamma \varphi_i \rangle^2.
\end{align*}
This means
\begin{align}
\sum_i \lambda_i^2  \langle \varphi_i, \Gamma \varphi_i \rangle^2 &\le 
 \trace^2\left( \sum_i \lambda_i \langle \varphi_i, \Gamma \varphi_i \rangle \varphi_i \otimes \varphi_i \right)\nonumber
 = \left(\sum_i \lambda_i \langle \varphi_i, \Gamma \varphi_i \rangle \langle \varphi_i, \varphi_i \rangle\right)^2\nonumber\\
 &=\left(\sum_i \lambda_i \langle \varphi_i, \Gamma \varphi_i \rangle\right)^2 =\trace^2\left( \Gamma C \right),\label{Eq:temp2}
\end{align}
Combining \eqref{Eq:temp1} and \eqref{Eq:temp2} in \eqref{Eq:tempo} yields the result. 
We will now verify \eqref{Eq:fubini}. 
Note that
\begin{align*}
&~\sum_{i,j,k,\ell} \E |x_ix_jx_k x_\ell| |\langle \varphi_i, \Gamma \varphi_j \rangle| |\langle \varphi_k, \Gamma \varphi_\ell \rangle|\le\Vert\Gamma\Vert^2 \sum_{i,j,k,\ell} \E |x_ix_jx_k x_\ell|\\
\lesssim&~ \| \Gamma\|^2 \left( \sum_i \E[x_i^4] + \sum_{i\neq j} \E[|x_i|^3] \E[|x_j|]+\sum_{i \neq j}  \E[x_i^2] \E [x_j^2]\right.\\
&~ \left.\qquad\qquad\qquad+  \sum_{i \neq j \neq k \neq \ell} \E[|x_i|] \E[|x_j|] \E[|x_j|]\E[|x_\ell|] \right)\\
\le&~ \| \Gamma\|^2 \left( 3\sum_i\lambda_i^2  + \frac{4}{\pi} \sum_{i\neq j} \lambda_i^{3/2} \sqrt{\lambda_j} +  \sum_{i\ne j} \lambda_i\lambda_j+ \frac{4}{\pi^2} \left( \sum_i \sqrt{\lambda_i}\right)^4 \right)\\
\lesssim&~ \| \Gamma\|^2 \left( \sum_i\lambda_i^2  + \left(\sum_{i} \lambda_i^{3/2} \right) \left(\sum_j \sqrt{\lambda_j}\right)+  \left(\sum_{i} \lambda_i\right)^2+ \left( \sum_i \sqrt{\lambda_i}\right)^4 \right)\\
<&~ \infty,
\end{align*}
which completes the proof.
\end{proof}

\begin{lemma}\label{lem:theorem2relatedlemma2}
For $\beta \geq \alpha >1$, and $\gamma \geq \frac{\alpha}{\beta}$, we have
\begin{align*}
\sum_{i \in \mathbb{N}} \frac{i^{-\alpha}}{(i^{-\beta} + \lambda)^\gamma} \leq \lambda^{-\tfrac{1+\beta\gamma -\alpha}{\beta}} 2^{1-\tfrac{\alpha}{\beta}} \int_0^\infty  \frac{1}{1+y^\alpha}\,dy.
\end{align*}
\end{lemma}

\begin{proof}
Note that 
\begin{align*}
&\sum_{i \in \mathbb{N}} \frac{i^{-\alpha}}{(i^{-\beta} + \lambda)^\gamma} = \sum_{i\in\mathbb{N}}   \frac{i^{\gamma\beta - \alpha}}{(1+ \lambda i^{\beta})^\gamma} \leq \int_0^\infty \frac{x^{\beta \gamma -\alpha}}{(1+\lambda x^\beta)^\gamma}\,dx \\
&= \int_0^\infty \frac{\lambda^{-\tfrac{(\beta\gamma - \alpha)}{\beta}} y^{\beta\gamma -\alpha}}{(1+y^\beta)^\gamma} \lambda^{-1/\beta}\, dy 
 = \lambda^{-\tfrac{(1+\beta\gamma - \alpha)}{\beta}} \int_0^\infty \frac{y^{\beta\gamma - \alpha}}{(1+y^\beta)^\gamma}\, dy,
\end{align*}
where
\begin{align*}
\int_0^\infty \frac{y^{\beta\gamma - \alpha}}{(1+y^\beta)^\gamma}\, dy  & = \int_0^\infty \frac{y^{\beta \left(\gamma - \frac{\alpha}{\beta}\right)}}{(1+y^\beta)^\gamma}\, dy  = \int_0^\infty \left( \frac{y^\beta}{1+y^\beta} \right)^{\gamma  - \frac{\alpha}{\beta}} \left(1+y^\beta \right)^{-\frac{\alpha}{\beta}}\, dy\\&\leq \int_0^\infty (1+y^\beta)^{-\frac{\alpha}{\beta}}\, dy\qquad\qquad(\because\,\, \gamma\ge\alpha/\beta).
\end{align*}
Therefore,
\begin{align*}
\int_0^\infty \frac{1}{(1+y^\beta)^{\alpha/\beta}}\, dy &= 2^{-\alpha/\beta} \int_0^\infty {\left(\frac{1}{2} + \frac{y^\beta}{2}\right)^{-\alpha/\beta}}\, dy
\leq 2^{-\alpha/\beta} \int_0^\infty {\left(\frac{1}{2} + \frac{y^\alpha}{2}\right)^{-1}}\, dy,  
\end{align*}
where the last inequality uses the fact that 
\begin{align*}
{\left(\frac{1}{2} + \frac{y^\beta}{2}\right)^{\alpha/\beta}} \geq {\left(\frac{1}{2} + \frac{y^\alpha}{2}\right)}
\end{align*}
by Jensen's inequality, because the function $f(x) = x^{\alpha/\beta},\,x\in[0,\infty)$ is concave for $\alpha \leq \beta$. Hence, we have
\begin{align*}
\int_0^\infty \frac{y^{\beta\gamma -\alpha}}{(1+y^\beta)^\gamma}\, dy \leq 2^{1-\frac{\alpha}{\beta}} \int_0^\infty (1+y^\alpha)^{-1} \, dy < \infty
\end{align*}
as long as $\alpha >1$. 
\end{proof}

\begin{lemma}\label{lem:theorem2relatedlemma1}
For any $0 \leq \alpha\leq \beta$,
\begin{align*}
\sup_{i\in\mathbb{N}}\left[ \frac{i^{-\alpha}}{i^{-\beta} +\lambda}\right] \leq \lambda^{\tfrac{\alpha -\beta}{\beta}}.
\end{align*}
\end{lemma}

\begin{proof}
Note that
\begin{align*}
&\sup_{i\in\mathbb{N}}\left[ \frac{i^{-\alpha}}{i^{-\beta} +\lambda}\right]  = \sup_{i\in\mathbb{N}}\left[ \frac{i^{\beta-\alpha}}{1+ \lambda i^{\beta} }\right] \leq \sup_{x\in(0,\infty)} \frac{x^{\beta - \alpha}}{1+\lambda x^\beta} = \sup_{t \in(0,\infty)} \frac{(t/\lambda)^{\tfrac{\beta - \alpha}{\beta }}}{1+t}\\
&=  \lambda^{\tfrac{\alpha - \beta}{\beta }}  \sup_{t \in(0,\infty)} \frac{t^{\tfrac{\beta - \alpha}{\beta }}}{1+t}
 =  \lambda^{\tfrac{\alpha - \beta}{\beta }}  \sup_{t \in(0,\infty)} \left( \frac{t}{1+t}\right)^{\tfrac{\beta - \alpha}{\beta }} \frac{1}{(1+t)^{1 - \frac{(\beta-\alpha)}{\beta}}}\\
 &\leq  \lambda^{\tfrac{\alpha - \beta}{\beta }}  \sup_{t \in(0,\infty)} \frac{1}{(1+t)^{\alpha/\beta}} =  \lambda^{\tfrac{\alpha - \beta}{\beta }},
\end{align*}
which completes the proof.
\end{proof}

\begin{lemma}\label{lem:interpretlemmatemp}
$$\int_0^1 \cos(ax) \cos(bx) dx = \frac{b}{b^2-a^2} \cos(a) \sin(b) - \frac{a}{b^2-a^2} \sin(a) \cos(b).$$
\end{lemma}
\begin{proof}
Define $J=\int_0^1 \cos(ax) \cos(bx) dx$. Then, we have
\begin{align*}
    J&=\left(\cos(ax) \frac{\sin(bx)}{b} \right)_0^1 + a \int_0^1 \sin(ax) \frac{\sin(bx)}{b} dx\\
    &=\frac{1}{b} \cos(a)\sin(b) +\frac{a}{b}\left[ \left(\sin(ax)\frac{\cos(bx)}{b} \right)_1^0+\frac{a}{b}\int_0^1 \cos(ax) \cos(bx) dx\right]\\
    &=\frac{1}{b} \cos(a) \sin(b) -\frac{a}{b^2}\sin(a)\cos(b) + \left(\frac{a}{b}\right)^2 J
\end{align*}
from which we get the desired result. 
\end{proof}
\begin{lemma}\label{lem:haar}
Let $\psi_m$ be as defined in~\eqref{eq:haar}. Then, we have
\[
\langle\psi_m(\cdot), \cos(i\pi\cdot) \rangle_{L^2} \leq \frac{4}{i\pi} 2^{\floor{\log_2 m}/2}.
\]
\end{lemma}
\begin{proof}
For a given $m$, consider $j=\floor{\log_2 m}$ and $\ell=m+1-2^j$. Then, we have
\begin{align*}
&\langle\psi_m(\cdot), \cos(i\pi\cdot) \rangle_{L^2} \\
=& \int_0^1 \psi_m(x)\cos(i\pi x) dx \\
=& \int_{\frac{\ell-1}{2^j}}^{\frac{\ell-1/2}{2^j}}2^{j/2}\cos(i\pi x) dx - \int_{\frac{\ell-1/2}{2^j}}^{\frac{\ell}{2^j}} 2^{j/2} \cos(i\pi x) dx\\
 =& \frac{2^{j/2}}{i\pi} \left[\sin(i\pi x) \right]_{\frac{\ell-1}{2^j}}^{\frac{\ell-1/2}{2^j}} - \frac{2^{j/2}}{i\pi} \left[\sin(i\pi x) \right]_{\frac{\ell-1/2}{2^j}}^{\frac{\ell}{2^j}}\\
=&\frac{2^{j/2}}{i\pi} \left[ \sin\left( \frac{i\pi(\ell-1/2)}{2^j} \right) -\sin\left( \frac{i\pi(\ell-1)}{2^j} \right) \right]\\
&\qquad\qquad- \frac{2^{j/2}}{i\pi} \left[\sin\left( \frac{i\pi\ell}{2^j} \right) -\sin\left( \frac{i\pi(\ell-1/2)}{2^j} \right) \right] \\
=&\frac{2^{j/2}}{i\pi}\left[ 2\sin\left( \frac{i\pi(\ell-1/2)}{2^j}\right) - \sin\left( \frac{i\pi(\ell-1)}{2^j}\right) - \sin\left( \frac{i\pi\ell}{2^j}\right) \right] \\
\leq& \frac{4}{i\pi}2^{j/2},
\end{align*}
thereby completing the proof. 
\end{proof}

\section{Bound on $\eta_{ij}$ in \eqref{Eq:etaij}}\label{subsec:etabound}
From \eqref{Eq:etaij}, we have
\begin{align*}
    \eta_{ij} = \sum_m b_m \theta_{mi}\theta_{mj} = \frac{1}{\pi^2}\sum_mb_m \frac{\omega_m}{\omega_m^2-i^2} \frac{\omega_m}{\omega_m^2-j^2} \sin^2(\pi \omega_m) (-1)^{i+j}
\end{align*}
which implies
\begin{align*}
    |\eta_{ij}| \lesssim \sqrt{\sum_m\left(\frac{\omega_m \sqrt{b_m}}{\omega_m^2-i^2}\right)^2} \sqrt{\sum_m\left(\frac{\omega_m \sqrt{b_m}}{\omega_m^2-j^2}\right)^2}.
\end{align*}
Consider
\begin{align*}
\sum_m\left(\frac{\omega_m \sqrt{b_m}}{\omega_m^2-j^2}\right)^2&=
    \sum_m \frac{\omega^2_m b_m}{(\omega_m+j)^2(\omega_m-j)^2} \leq \frac{1}{j^2} \sum_m \left(\frac{\omega_m}{\omega_m-j} \right)^2 b_m\\ &\lesssim \frac{1}{j^2} \sum_m m^{-(1+\delta)} \left( \frac{am+b}{am+b-j} \right)^2 \\
    &= \frac{1}{j^2}\sum_m m^{-(1+\delta)} \left( \frac{m+b/a}{m+(b-j)/a} \right)^2.
\end{align*}
We now consider two cases. First, we consider the case of \textbf{$b<0$}. Define $c\coloneqq (j-b)/a -\floor{(j-b)/a}$ and note that $0 <c <1$. Then, we have 
\begin{align}
    &\sum_m m^{-(1+\delta)} \left( \frac{m+b/a}{m+(b-j)/a} \right)^2\nonumber\\ &\lesssim \sum_m m^{-(1+\delta)} \left(\frac{m}{m-(j-b)/a} \right)^2
    = \sum_m \frac{m^{1-\delta}}{\left( m - (j-b)/a \right)^2}\nonumber\\
    &= \sum_{m\leq \floor{\frac{(j-b)}{a}}}  \frac{m^{1-\delta}}{\left( m - (j-b)/a \right)^2} + \sum_{m= \floor{\frac{(j-b)}{a}}+1}^\infty  \frac{m^{1-\delta}}{\left( m - (j-b)/a \right)^2}\nonumber\\
    &= \sum_{m\leq \floor{\frac{(j-b)}{a}}}  \frac{m^{1-\delta}}{\left( (j-b)/a-m \right)^2} + \sum_{m= \floor{\frac{(j-b)}{a}}+1}^\infty  \frac{m^{1-\delta}}{\left( m-(j-b)/a \right)^2}\nonumber\\
    &= \left[\frac{\floor{\frac{j-b}{a}}^{1-\delta}}{c^2} + \frac{\floor{\frac{j-b}{a}-1}^{1-\delta}}{(1+c)^2} + \cdots + \frac{1}{\left(\frac{j-b}{a} -1 \right)^2}  \right]\nonumber\\
    &\qquad\qquad+ \left[ \frac{\left( 1+\floor{\frac{j-b}{a}} \right)^{1-\delta}}{(1-c)^2} + \frac{\left( 2+\floor{\frac{j-b}{a}} \right)^{1-\delta}}{(2-c)^2} +\cdots\right]\nonumber\\
    &\leq  \left\lfloor{\frac{j-b}{a}}\right\rfloor^{1-\delta} \left(\frac{1}{c^2} + 1+ \frac{1}{2^2}+\frac{1}{3^2}+\cdots \right) +\frac{\left( 1+\floor{\frac{j-b}{a}} \right)^{1-\delta}}{(1-c)^2}\nonumber\\
    &\qquad\qquad\qquad+ \sum^\infty_{\ell=1}\frac{\left( \ell+1+\floor{\frac{j-b}{a}} \right)^{1-\delta}}{\ell^2}
    \nonumber\\
    &\leq\left\lfloor{\frac{j-b}{a}}\right\rfloor^{1-\delta} \left(\frac{1}{c^2}+\frac{\pi^2}{6} \right)+ \begin{cases}
    \frac{1+\floor{\frac{j-b}{a}}^{1-\delta}}{(1-c)^2} +\sum^\infty_{\ell=1}\frac{\left( \ell+1\right)^{1-\delta}+\floor{\frac{j-b}{a}}^{1-\delta}}{\ell^2},~\text{for}~\delta\leq 1\nonumber\\
    \frac{1}{(1-c)^2} + \sum^\infty_{\ell=1}\frac{1}{\ell^2},~\text{for}~\delta >1
    \end{cases}\nonumber\\
    &\lesssim\left\lfloor{\frac{j-b}{a}}\right\rfloor^{1-\delta} + C_1,\nonumber
\end{align}
where we have used the fact that for $\delta \leq 1$, 
\[
\sum_{\ell=1}^\infty \frac{(\ell+1)^{1-\delta}}{\ell^2} \leq \sum_{\ell=1}^\infty \frac{\ell^{1-\delta}+1}{\ell^2} = \frac{\pi^2}{6} + \sum_{\ell=1}^\infty\frac{1}{\ell^{1+\delta}} < \infty,
\]
and 
$C_1$ is some constant independent of the index $j$. Hence
\begin{align*}
\sum_m \frac{\omega^2_m b_m}{(\omega_m+j)^2(\omega_m-j)^2}  &\lesssim \frac{1}{j^2} \left[ \left\lfloor{\frac{j-b}{a}}\right\rfloor^{1-\delta} + C_1 \right]\leq \frac{1}{j^2} \left[ \left(\frac{j-b}{a} \right)^{1-\delta}+C_1\right]\\ &\lesssim j^{-\min(1+\delta,2)}.
\end{align*}
Next we consider the case of \textbf{$b>0$}. Note that for $j<b$, 
\begin{align*}
\sum_m m^{-(1+\delta)} \left( \frac{m+b/a}{m+ (b-j)/a} \right)^2 &\leq \sum_m m^{-(1+\delta)} \left(\frac{m+b/a}{m} \right)^2\\ 
&= \sum_m m^{-(1+\delta)} \left(1+\frac{b}{am} \right)^2\\
& \lesssim \sum_m m^{-(1+\delta)} \left[ 1+ \frac{1}{m^2}\right] \coloneqq C_2 < \infty.
\end{align*}
Now, for $j>b$, we have
\begin{align*}
    &\sum_m m^{-(1+\delta)} \left( \frac{m+b/a}{m+(b-j)/a}\right)^2= \sum_m m^{-(1+\delta)}\frac{\left( m+b/a\right)^2}{\left(m - (j-b)/a
    \right)^2}\\
    &\lesssim \sum_m m^{-(1+\delta)} \frac{m^2+(b/a)^2}{\left(m - (j-b)/a
    \right)^2}\\
    &=\sum_m \frac{m^{1-\delta}}{\left( m - (j-b)/a \right)^2} + \left(\frac{b}{a}\right)^2 \sum_m \frac{m^{-(1+\delta)}}{\left(m-(j-b)/a\right)^2}\\
    &\leq \left[1+\left(\frac{b}{a} \right)^2 \right] \sum_m \frac{m^{1-\delta}}{\left( m-(j-b)/a\right)^2}\lesssim \left\lfloor{\frac{j-b}{a}}\right\rfloor^{1-\delta}+C_1.
\end{align*}
Therefore, 
\begin{align*}
     \sum_m m^{-(1+\delta)} \left( \frac{m+b/a}{m+(b-j)/a}\right)^2 \lesssim \begin{cases}
     C_2,~\text{for}~ j <b,\\
     \floor{\frac{j-b}{a}}^{1-\delta}+C_1,~\text{for}~ j >b
     \end{cases},
\end{align*}
and consequently
\begin{align*}
    \sum_m \frac{\omega^2_m b_m}{(\omega_m+j)^2(\omega_m-j)^2}\lesssim \begin{cases}
    j^{-2},~\text{for}~j<b,\\
    j^{-\min(1+\delta,2)},~\text{for}~j>b
    \end{cases} \lesssim j^{-\min(1+\delta,2)}.
\end{align*}
Putting everything together yields
\[
\eta_{ij} = (ij)^{-\min\left(1, \tfrac{1+\delta}{2} \right)}.
\]


\end{document}